\documentclass[english,11pt]{amsart}

\usepackage[utf8]{inputenc}
\usepackage{amssymb,stmaryrd}
\usepackage{amsfonts}
\usepackage{amscd}
\usepackage{url}
\usepackage[all,cmtip]{xy}
\usepackage{color}
\usepackage{enumerate}
\usepackage{graphicx}
\usepackage{epstopdf} 
\usepackage{time}

\usepackage[T1]{fontenc}
\RequirePackage[colorlinks=true, breaklinks=true, urlcolor= green, linkcolor= blue, citecolor= green, bookmarksopen=true,linktocpage=true,plainpages=false,pdfpagelabels]{hyperref}

\definecolor{better-green}{rgb}{0,.6,.1}
\definecolor{violet}{rgb}{.5,0,.5}

\newcommand{\CC}{\mathbb{C}}

\newcommand{\NN}{\mathbb{N}}

\newcommand{\RR}{\mathbb{R}}

\newcommand{\VP}{\mathrm{VP}}
\newcommand{\VNP}{\mathrm{VNP}}

\newtheorem{thm}{Theorem}[section]
\newtheorem{cor}[thm]{Corollary}
\newtheorem{prop}[thm]{Proposition}
\newtheorem{lem}[thm]{Lemma}

\theoremstyle{definition}
\newtheorem{defn}[thm]{Definition}
\newtheorem{notn}[thm]{Notation}

\newtheorem{rem}[thm]{Remark}
\newtheorem{exam}[thm]{Example}
\newtheorem{rems}[thm]{Remarks}

\newcommand*\latinnumeral[1]
  {\ifcase#1\unskip\or\unskip\or bis\fi}
\newcommand*\latintags
  {\xdef\startingequationnumber{\the\numexpr\value{equation}+1\relax}%
   \setcounter{equation}{0}%
   \aftergroup\resetarabicequations
   \def\theequation
     {\startingequationnumber~\textit{\latinnumeral{\value{equation}}}}}
\newcommand*\resetarabicequations
  {\setcounter{equation}{\numexpr\startingequationnumber\relax}}

\begin{document}

\title{rational and lacunary algebraic curves}
\author{Georges Comte}
\address{Univ.   Savoie Mont Blanc, CNRS, LAMA, 73000 Chamb\'ery, France}
\email{georges.comte@univ-smb.fr}
\urladdr{https://georgescomte.perso.math.cnrs.fr/}
\author{Sébastien Tavenas}
\address{Univ.   Savoie Mont Blanc, CNRS, LAMA, 73000 Chamb\'ery, France}
\email{sebastien.tavenas@univ-smb.fr}
\urladdr{https://tavenas.perso.math.cnrs.fr/}
\thanks{}

\maketitle

\begin{abstract}
We give a bound on the number $\mathcal{Z}$ of intersection points in a ball of the complex plane, between a rational curve and a lacunary algebraic curve $Q=0$. This bound 
depends only on the lacunarity diagram of $Q$, and in particular is uniform in the coefficients of $Q$. 
Our bound shows that $\mathcal{Z}=O(dm)$, where 
$d$ is the degree of $Q$ and $m$ is the number of its monomials. 
\end{abstract}

\tableofcontents

\section*{Introduction}

By the classical Bézout theorem (see \cite{Bezout} for the original article and, for instance, \cite{BCSS}) for any positive integer $n$, the number of isolated complex solutions of a polynomial system $F_1=\cdots = F_n=0$, $F_1, \ldots, F_n\in \CC[X_1, \ldots, X_n]$, is at most the product $d_1\cdots d_n$ of the degrees of the $F_i$'s.
Requiring the polynomials have prescribed support, one can improve on this bound: for instance in case all polynomials $F_1, \ldots, F_n$ have the same given support $S\subset \NN^n$, by Kushnirenko's theorem the number of isolated complex solutions in $(\CC^*)^n$ is at most $n!\mathrm{Vol}_n(\Delta_S)$, where $\Delta_S$ is the convex hull of $S$ in $\RR^n$ (see \cite{Kush1}). More accurately, in case $S_i$ is the support of $F_i, i=1, \ldots,n$, by Bernstein's theorem, often called BKK's theorem, the number of isolated complex solutions in $(\CC^*)^n$ is at most the Minkowski mixed volume of $\Delta_{S_1}, \ldots , \Delta_{S_n}$ (see \cite{Bernstein, Kush1, Kush}). In some sense those improvements allow to measure how far
the system is from the simplest situation where it can be solved by composition in order to 
eliminate one by one the variables, to eventually give a univariate polynomial of degree  $d_1\cdots d_n$. 

The case of real polynomial systems and their real solutions has been extensively investigated in particular since the seminal work  of A. Khovanski\u{\i}, (see \cite{KhoFew}, and for instance \cite{Sottile} for an overview, and \cite{Kushletter} for historical information), although some basic problems still resist, even in the lowest degree. For instance, there is still no answer to the question: for a fixed number $n\ge 2$ of variables, is the number of isolated solutions 
of the system polynomially bounded with respect to the sizes of the supports of the polynomials  $F_1$ and $F_2$
(see for instance \cite{KoiPoTa} on this question, for $n=2$)?

We consider in what follows the complex bivariate case, that is the case of the intersection of two complex algebraic plane curves. More specifically we consider the case of a given and fixed rational curve $(P_1,P_2)$ in $\CC^2$ (containing the origin of $\CC^2$) and of an algebraic curve $Q(X,Y)=0$, the coefficients $\lambda\in \CC^m$ and the degree $d$ of $Q$ being parameters of the problem. Actually, we are interested in the asymptotics of bounds on the number of intersection points of the two curves, uniformly with respect to $\lambda$, as $d\to +\infty$ (see Example \ref{exam.Bounds} and Remark \ref{rem.faster lacunarity}). 

For simplicity, assume here that $P_1,P_2\in \CC[X]$. Then, of course in this situation, since one can compose the rational parametrization $(P_1,P_2)$ with $Q$, the number of intersection points of the two curves $(P_1,P_2)(\CC)$ and $Q(X,Y)=0$ in $\CC^2$ is simply the degree of the univariate polynomial $f=Q(P_1,P_2)$, whatever the specific forms of $P_1, P_2$ and $Q$ are. 
To obtain nontrivial bounds we thus count the number of intersection points only in a ball of $\CC^2$ centred at the origin and of fixed radius, 
for instance in a ball $B(0,\rho)$, $\rho>0$, such that $(P_1,P_2)^{-1}(B(0,\rho))\subset D_{R_\rho}$, for some $R_\rho>0$. Such a positive number $R_\rho$ exists in the case $P_1$ and $P_2$ are (non constant) polynomials, since $\vert P_i (z) \vert \to +\infty$ when $\vert z \vert\to +\infty$. 
We then consider, instead of parameters in $D_{R_\rho}$, parameters in the unit disc $D_1$ (see Remark \ref{rem.R=B=1}). In the general rational case, we may have $(P_1,P_2)(z)\to 0$ when $\vert z \vert \to +\infty$. In this case we compute the additional number of intersection points between $Q(X,Y)=0$ and the branch of the rational curve $(P_1,P_2)$ corresponding to parameters going to infinity, in the same way, using the transformation $z\mapsto 1/z$ of the parameters. 

Denoting $\max(\deg(P_1), \deg(P_2))$ by $D$, the number of zeros $\mathcal{Z}=
\mathcal{Z}_{P_1,P_2,Q}$ of $f$ in $\bar{D}_1$ may still be $dD$, the maximal expected number of zeros, depending on the choice of $P_1,P_2$ and $Q$ (see Remark \ref{rem.dependance en a0}). On the other side, considering the maximal multiplicity $b$ of $f=f_\lambda$ at the origin, with respect to all possible choices of coefficients $\lambda\in \CC^m$ of $Q$ such that $f$ is not identically $0$ ($b$ is called the \emph{Bautin index of the family $f_\lambda$}, see Definition \ref{de: Bautin index} and Proposition \ref{prop:max multiplicity and Bautin index}), any small variation of the constant term of $Q$ will create $b$ distinct roots of $f$, as closed as desired from the origin. In consequence, the question of a bound for $\mathcal{Z}$, uniform with respect to the coefficients $\lambda$ of $Q$, is the question of how many zeros of $f$ we create in $\bar{D}_1$ in addition to the amount of $b$ imposed zeros, when freely moving the parameters of the system, and how far can this number be under $dD$.  

We answer this question for lacunary curves $Q(X,Y)=0$: we show in Theorem \ref{thm.Bezout Bound for lacunary polynomials} that $\mathcal{Z}$ is bounded from above by an expression
depending on the \emph{lacunarity diagram of $Q$} (see Definition \ref{def.lacunarity diagram}) and on $\alpha_0, a_0$, the initial coefficients of $P_1$ and $P_2$ (Remark \ref{rem.dependance en a0} shows that the dependency on some coefficients of $P_1$ and $P_2$ of our bound is unavoidable). This bound lies in  $]b,dD[$, with $b<dD$, in case
$Q$ is sufficiently sparse and $\alpha_0$ and $a_0$ are not too closed to~$0$. 

The lacunarity of $Q$ is a necessary condition for having nontrivial bounds  on $\mathcal{Z}$, and the reason is easy to explain. Indeed, $m$ being the dimension of the space of parameters of $Q$, by solving $\min(m-1,dD)$ linear equations with $m$ variables, one can cancel the first $\min(m-1,dD)$ monomials of $f$, and thus $b\ge \min(m-1,dD)$. It follows that $m$, which could be at most $(d+1)(d+2)/2$, has in fact to be less than $dD$ when one searches better bounds than $dD$ for $\mathcal{Z}$. 

It is worth noticing, on the other side, that contrary to what happens in the real case, $\mathcal{Z}$ cannot depend only on the number of monomials of $P$ and $Q$, but has to depend also on $d$.
 Indeed, the lacunarity conditions for $Q$ we consider (see conditions 
 \eqref{eq.Lacunarity Condition 1},  \eqref{eq.Lacunarity Condition 1bis},  \eqref{eq.Lacunarity Condition 1ter} and  \eqref{eq.Lacunarity Condition 1ter with mult}
 in  Section \ref{Section 5}) are based on a fast growth of the degrees of the homogeneous parts of $Q$ (such as geometric progressions) and allow us to compute%
 \footnote{This computation is otherwise generally not tractable, specially in the analytic case. Our lacunarity conditions guarantee here that for any $\lambda\in \CC^m$, the two curves have isolated intersections points.}
 the exact value of $b$ (see Proposition \ref{prop.valeur de b}), showing that 
 $b\ge d\min(\nu_1,\nu_2)$ (where \(\nu_1\) and \(\nu_2\) are the multiplicities of \(P_1\) and \(P_2\) respectively in \(0\)). This lower bound for $b$ in turn shows that $\mathcal{Z}$ cannot be independent on $d$, whatever the number of monomials $P$ may have. 
 
 This observation leaves little scope between $b$ and $dD$ to find better bounds than $dD$; A goal which is however achieved in our main Theorem~\ref{thm.Bezout Bound for lacunary polynomials}. For instance  (see Remark \ref{rem.faster lacunarity}),  for $d$ large enough a condition allowing $\mathcal{Z}\in ]b,dD[$ is 
 $$
\tau\log \left( 
   \dfrac{1}{ \vert \alpha_0 a_0\vert }  \right)\le \dfrac{D}{5}-\max(\nu_1,\nu_2)\log2,
$$
 where $\tau$ is the maximal amount of monomials of degree $d$ in $Q$. 
 
 Notice in particular that our lacunarity conditions allow in such examples (Example \ref{exam.Bounds} and Remark
 \ref{rem.faster lacunarity})  the total number of monomials in $Q$ to be unbounded with respect to the degree of $Q$, contrary to the lacunarity conditions usually considered in the theory of fewnomials. 
Moreover the dependency of our bounds in the number of monomials of $Q$ is not exponential, contrary to the case of most of the bounds so far obtained for polynomial systems with a fixed number of monomials (see again \cite[Section 6.2]{Sottile} for bounds in the real case). 
 In the real plane, a polynomial bound in $\deg(F)$ and $t$ is obtained in~\cite{KoiPoTa} for the number of solutions of the polynomial system $F(X,Y)=G(X,Y)=0$, with $G$ having $t$ monomials. This bounds turns out to be $O(\deg(F)^3t+\deg(F)^2t^3)$, in contrast to the bounds obtained in this article that may be linear in $d$ and $\tau$ (see \eqref{exam.bound1} and \eqref{exam.bound2}). 
 
 One can next wonder
whether, for two given algebraic curves \(Q_1(X,Y)=0\) and \(Q_2(X,Y)=0\), it is possible to bound the number of their intersection points in a small ball (centred at a point outside the coordinate axes\footnote{The multiplicity can be made arbitrarily large on the coordinate axes without altering the number of monomials by multiplying \(Q_1\) and \(Q_2\) by a large monomial.}) by a polynomial function in the number of monomials of \(Q_1\) and \(Q_2\). This would imply in particular a bound on the multiplicity of the roots outside the coordinate axes. In this direction a polynomial bound of the mixed form \(O(\deg(F)^2t^2)\) is obtained in~\cite{KS20} for the multiplicity of an isolated solution (with nonzero coordinates) of the system \(F(x,y)=G(x, y)=0\), with $G$ having again $t$ monomials. Moreover, this would bring us closer to proving the \emph{\(\tau\)-conjecture for multiplicities} defined in~\cite{KS20} (Hrube\v{s} noticed~\cite{Hru13} that it derives from the \emph{real \(\tau\)-conjecture} introduced in~\cite{Koi11}).
This family of questions and conjectures on lacunary bivariate polynomials are deep, since for instance
both the real $\tau$ and  $\tau$-conjectures have dramatic consequences in algebraic complexity; They imply almost the separation of both main algebraic complexity classes \(\VP\) and \(\VNP\) (for an introduction to the \(\VP\) versus \(\VNP\) see for example~\cite{SY10,Bur00}).

An other phenomenon of importance, arising from part $B$ of Hilbert’s 16th problem, is the question of the accumulation of limit cycles of trajectories of planar polynomial vector fields. In this context, the Bautin index of the difference, between the
Poincaré first return map of a polynomial deformation of the planar field and the identity map, is a bound on the number of cycles of the vector field. Localizing those cycles therefore consists in knowing the maximal radius of the disc where only $b$ zeros occur, the number of zeros that one necessarily finds, as observed above, in any ball centred at zero, say (see for instance among many other references \cite{Bau1, Bau2, FraYom, HTR, Rou, Yako}). 
We provide a lower bound for this radius in our algebraic context, that is, we give in Theorem \ref{thm.Bezout Bound for lacunary polynomials} a lower bound in terms of the lacunarity diagram of $Q$, $\alpha_0$ and $a_0$, for the radius $\rho$ of a disc $\bar{D}_\rho\subset \bar{D}_1$ in which $f_\lambda$ has only $b$ zeros, uniformly in $\lambda$. 

In \cite[Section 3.4]{CoYo1} the case of a transcendental analytic lacunary curve $(x,h(x))$ has been considered, in order to obtain a  bound for the number of zeros of $Q(x,h(x))$ on $\bar{D}_1$, polynomial in $d$. As explained above, we consider instead here the quite opposite case where the curve is rational and the polynomial family $Q$ is lacunary. The challenge being here to find some room between $b$ and $dD$ for $\mathcal{Z}$, and we indeed show that there is some. In both cases (analytic in \cite{CoYo1} and rational here), we start from a Jensen-Nevanlinna type estimate of zeros in Bernstein classes (see \cite[Section 2.2]{Roy.Yom}, \cite{Yako}, and Remark \ref{rem.Jensen}), all the information we need to obtain an estimation of $\mathcal{Z}$ is contained in the \emph{Bautin matrix $M$} of $f_\lambda$ (see Section \ref{section 4}), defined as the matrix of the linear map sending the parameters $\lambda$ of $Q$ to the coefficients of the polynomial $f_\lambda$.  
In particular $b$ appears as the rank of the lowest row in $M$ independent of the above rows, and $\mathcal{Z}$ may be estimated from $b$, the module $\delta$ of a non zero minor of maximal size in $M$, and from the dimension $\sigma$ of the Bautin ideal. Consequently our method consists in the reduction of $M$: we provide in Section \ref{section 2} and \ref{section 3} an explicit Gau\ss{} reduction of $M$, that we can perform by blocks thanks to our lacunarity conditions. Crucial technical tools to achieve the reduction are the combinatorial Lemmas \ref{lem.Formule combinatoire} and \ref{lem.Formule combinatoire multivariee}.
A special case occurs in this reduction,  enlightening a geometrical degeneration condition on the parametrization $(P_1,P_2)$, the case where $P_1$ and $P_2$ have proportional initial jets at the origin, up to some non trivial order $>1$. The reduction of the Bautin matrix in the singular case is specifically treated in Section \ref{section 3}. In this singular case, $b$ and $\delta$ differ from their values in the regular case, where $P_1$ and $P_2$ have no proportional jets.     

In Section \ref{Section6} we show how to reduce the general case of a  rational curve to the case of $(P_1,P_2)(\CC)$, with $P_1,P_2\in \CC[X]$, under our lacunarity conditions, and therefore from now on and up to the final Section \ref{Section6}, we only consider the polynomial case.

\section{Intersection of a curve and a family of curves}\label{section 1}
For the convenience of the reader, we recall here the notation of \cite{Yom,CoYo1}, in order to adapt it to our context.

In what follows, for a real number $R>0$, $\bar{D}_R$ is the closed disc $\{ z\in  \CC, \vert z\vert \le R \}$ in $\CC$. 
For $D_1, D_2\in \NN\setminus \{0\}$ and 
complex numbers $a^i_0,\ldots, a_{D_i}^i$, 
$i=~1, 2$, we consider
\begin{equation}\label{eq.Polynomial P}
P(X)=\left(P_1(X), P_2(X)\right)=\left(\sum_{k=0}^{D_1} a^1_kX^k,\sum_{k=0}^{D_2} a^2_kX^k\right),
\end{equation}
a nonzero polynomial mapping of degree $D=\max\{D_1, D_2\}$, sending $0\in \CC$ to $0\in \CC^2$ (thus $a^1_0=a^2_0=0$), and we denote by $M_R$ the real number
\begin{equation}\label{eq.Max P}
M_R=\max_{z\in \bar{D}_R, i=1,2}\vert P_i(z) \vert.
\end{equation}

Let $m\in \NN$, and $Q_i$, $i=1,\ldots,m$, be a family of monomials 
of $\CC[X,Y]$ of degree at most $d$, for some integer $d\in \NN$.
For $\lambda=(\lambda_1,\ldots,\lambda_m)\in\CC^m$, we denote by $Q_\lambda$ the following polynomial of degree at most $d$
of $\CC[X,Y]$
$$ Q_{\lambda}=\sum^m_{i=1}\lambda_iQ_i.$$

In what follows we are interested in the set of zeros in $\bar{D}_R$ (actually in $\bar{D}_{R/4}$) of the linear family (with parameters $\lambda_1, \ldots, \lambda_m$) of polynomial functions 
\begin{equation}\label{eq:f.lambda}
f_{\lambda}(z)=Q_{\lambda}(P(z))=\sum^m_{i=1}\lambda_i Q_i(P(z)).
\end{equation}

Geometrically this is the number of intersection points, for $z\in \bar{D}_R$, between the plane algebraic curves $z \mapsto P(z)$ of degree $D$ of $\CC^2$ and $Q_\lambda=0$ of degree $d$ of $\CC^2$.

\begin{rem}\label{rem.Null}
If $Q_\lambda(P(z))=0$ for any $\lambda\in \CC^m$ 
and any 
$z\in \bar{D}_R$, then for any $i=1, \ldots, m$, 
$Q_i(P)=0$ on $\bar{D}_R$. We will always consider families $Q_1, \ldots, Q_m$ such that for at least one parameter 
$\lambda\in \CC^m$, $f_\lambda(P)$ is not the zero polynomial.
We can easily achieve this condition for instance by taking the monomial of constant value $1$ in the family, or by allowing no zero polynomial among the $P_i$'s.
Moreover we will consider zeros of $f_\lambda$ only for such parameters, and in this case the number of zeros of $f_\lambda$ is a finite number. 
\end{rem}

Denoting, for $i=1, \ldots, m$, 
$$
Q_i(P(z))=\sum_{k=0}^{dD}c_k^iz^k, \ \  \mathrm{ and } \ \ 
B_R=\sup_{z\in \bar{D}_R, i=1, \ldots, m} \vert Q_i(P(z))\vert, 
$$
by Cauchy estimate we have for any $i=1, \ldots, m$ and any $k=0, \ldots, dD $
\begin{equation}\label{eq:curve2.1}
\vert c_k^i \vert \le \dfrac{B_R}{R^k}. 
\end{equation}
We write $f_\lambda$ as a polynomial with coefficients linear forms $v_0, \ldots, v_{dD}$ of parameters $\lambda_1, \ldots, \lambda_m$. Namely
\begin{equation}\label{eq:curve1}
f_\lambda(z)=\sum_{k=0}^{dD}v_k(\lambda)z^k, \ \ \mathrm{where} \ \ v_k(\lambda)=\sum^m_{i=1}c^i_k\lambda_i .
\end{equation}
We thus have by \eqref{eq:curve2.1}, for any $k=0, \ldots, dD$
\begin{equation}\label{eq:curve3}
\vert v_k(\lambda)\vert\leq \frac{m B_R |\lambda|}{R^k} ,
\end{equation}
where $|\lambda|=\max \{|\lambda_1|,\ldots, |\lambda_m| \}$.

Now we define the Bautin index $b=b(f_\lambda)$ of the family $f_\lambda$. 
 Let, for $i\in \NN$,  $L_i $ be the linear
subspace of $\CC^m$ defined by   $v_0(\lambda) = \ldots = v_i(\lambda)=0$. We have
$$
L_0\supseteq L_1\supseteq \ldots L_i\supseteq \ldots \supseteq L_{dD}.
$$
Hence on a certain step $b\le dD$ this sequence stabilizes
$$
L_{b-1}\supsetneqq L_b=L_{b+1}=\cdots =L_{dD}.
$$

\begin{defn}\label{de: Bautin index}
We call the integer $b$ defined above the \emph{Bautin index of the family $f_\lambda$} (see
\cite{Bau1,Bau2}).
\end{defn}

\begin{rems}
\begin{itemize}
\item[-] 
Note that $\lambda\in L_b \Longleftrightarrow v_k(\lambda)=0,$ for $k=0,1, \ldots, dD$. 
\item[-] Equivalently, $b$ is the biggest index $k$ such that $v_k\not\in \mathrm{Span}(v_0, \cdots, v_{k-1})$. 
\item[-] 
Following Remark \ref{rem.Null}, for a given polynomial mapping $P$, we choose a family $Q_i, i=1, \ldots, m$ such that $L_b\not=\CC^m$. 
\item[-] 
The curve $P(\CC)$ is an irreducible algebraic subset of $\CC^2$. The vector space $L_b$ is the space of parameters 
$\lambda\in \CC^m$ such that $Q_\lambda\in I$ where $I$ is the ideal  of $P(\CC)$ in $\CC[X,Y]$. 
\end{itemize}
\end{rems}

Since $L_b\not=\CC^m$, for $\lambda\not\in L_b$, the multiplicity of $f_\lambda(z)$ at the origin is well defined. 
The following proposition is a variant of \cite[Proposition 2.3]{CoYo1}. 
\begin{prop}\label{prop:max multiplicity and Bautin index}
Let us denote by $\mu$ the maximal multiplicity at the origin of $f_\lambda(z)$, with respect to the parameters $\lambda\not\in L_b$. Then $\mu=b\ge m-1$.
\end{prop}

\begin{proof}
   There exists a  parameter $\lambda\not\in L_b$ such that $f_\lambda(z)$
has multiplicity $\mu$, that is such that
$$f_\lambda(z)=v_\mu(\lambda) z^\mu +v_{\mu+1}(\lambda)z^{\mu+1}+\cdots, $$
with $v_\mu(\lambda)\not=0$.
Therefore $\lambda\in L_{\mu-1} \setminus L_\mu$. It follows that $L_\mu \subsetneqq L_{\mu-1}$, and  $b\ge \mu$. On the other hand,  since no parameter $\lambda\not\in L_b$ can cancel $v_0, \ldots, v_\mu$ in the same time,
$L_\mu\subseteq L_b$. But since $L_{b-1}\supsetneqq L_b=\cdots =L_{dD} $, we have $b\le \mu$.

Using this characterization of $b$ as $\mu$, the bound $b\ge m-1$ comes from the fact that the linear system
$v_0=\cdots= v_{m-2}$ with $m$ parameters and $m-1$ equations always has a nonzero solution.
\end{proof}

In Section \ref{section 4}, we bound from above the number of zeros of $f_\lambda$ in $\bar{D}_R$. For this, all the information we need is encoded in the $b+1$ first rows of 
the rank $\sigma$ matrix
$ M(f_{\lambda})=(c^i_k), \ k=0,\ldots,dD, \ i=1,\dots,m$, called the \emph{Bautin matrix of the family $f_\lambda$}.
With notation \eqref{eq:curve1}, the Bautin matrix $M(f_{\lambda})$ is defined by
$$  \begin{pmatrix}
v_0(\lambda) \\
\vdots \\
v_{dD}(\lambda)
\end{pmatrix}
= M_{\lambda}
 \begin{pmatrix}
\lambda_1 \\
\vdots \\
\lambda_m
\end{pmatrix}.
$$

\section{Triangulation of Bautin blocks}\label{section 2}
\label{s.triangulation}

 Since the $(k+1)$-th line of the Bautin matrix $M(f_{\lambda})$ is by definition the line of coefficients of the linear form $v_k$, by \eqref{eq:curve1} we have to compute the coefficients of $X^k$ in $f_\lambda(X)=Q_\lambda(P(X))$ to find this $(k+1)$-th line. 

For a monomial $X^i Y^j\in \CC[X,Y]$, of degree $i+j$, we denote by $\lambda_{i,j}$ the corresponding parameter coefficient in the family $Q_\lambda$. We also denote by $n_0 < \ldots < n_{\ell_d}=d$ the total degrees of monomials occuring in $Q_\lambda$. In the case where the family \(Q_\lambda\) corresponds to all monomials \(X^iY^{n_\ell-i}\) for all \(0 \le \ell \le \ell_d\), we simply denote by \(M\) the Bautin matrix of \(f_\lambda\). Notice that if we choose a subfamily for \(Q_\lambda\), the associated Bautin matrix is a submatirx of \(M\) (by selecting the columns corresponding to the selected monomials).

We fix in this section a degree $n_\ell$, $\ell\in \{0,\ldots, \ell_d\}$.  

\begin{rem}\label{rem.choix de colonnes}
 We first give hereafter the particular form 
 of the block, denoted $M_{n_\ell}$, appearing in the Bautin matrix $M$ as the contribution of all possible monomials of $X^I Y^{n_\ell -I}$, $I=0, \ldots, n_\ell$, of  fixed degree $ n_\ell$.  In general in our family $Q_\lambda(X,Y)$ only a certain amount of those monomials appears, this is why we next introduce (see Notation \ref{notn.Matrice avec choix de colonnes}) the block $M_{n_\ell,T}$ corresponding to a certain choice $T$ of $\tau\le n_\ell+1$ particular columns
 in $M_{n_\ell}$, that is, to a certain choice of $\tau$ parameters
 $\lambda_{I,n_\ell-I}$ in the family $Q_\lambda$. 
\end{rem}

\begin{rem}\label{rem.multiplicite plus grande que 1}
As already mentioned, up to some translation, we assume that $(P_1,P_2)(0)=0$, and therefore that the multiplicity at the origin of $P_1$ and $P_2$ is bigger than $1$. 
\end{rem}

As before, let \(D = \max(D_1,D_2)\). Let \(\nu_1\) and \(\nu_2\) be the multiplicities at the origin of \(P_1\) and \(P_2\) respectively. Up to permutation of the variables, we assume in what follows that \(\nu_1 \ge \nu_2 \ge 1\), and we will use the notation 
\[\nu = \min(\nu_1,\nu_2) = \nu_2.\]
We write $P_1$ and $P_2$ in the following way
$$P_1(X)=X^{\nu_1} \sum_{i=0}^{D_1-\nu_1}\alpha_i X^i, 
\ \ P_2(X)=X^{\nu_2} \sum_{i=0}^{D_2-\nu_2}a_i X^i,$$
and
$$
\ \ Q(X_1,X_2)=X_1^IX_2^{n_\ell-I},$$ 
for $I=0, \ldots, n_\ell$, where \(\alpha_0\) and \(a_0\) are nonzero. If needed, we will adopt the convention that coefficients above the degree are set to \(0\). 
Denoting 
$N_I=(n_\ell-I) D_2 +ID_1 - n_\ell$ and 
 $ 
       \displaystyle \sum_{i_1+\cdots+i_I+(\nu_1-\nu_2)I+j_1+\cdots+j_{n_\ell-I}=k }$ by $\displaystyle\sum_{*=k }$, for $k=0,\ldots, N_I $,
the parameter $\lambda_{I,n_\ell-I}$ in $f_\lambda$ appears in front of the polynomial 

$$
     P_1(X)^IP_2(X)^{n_\ell-I} =
     \sum_{k=0}^{ID_1+(n_\ell-I)D_2-\nu n_\ell}
    \sum_{*=k}
    \alpha_{i_1}\cdots\alpha_{i_I}
    a_{j_1}\cdots a_{j_{n_\ell-I}}X^{\nu n_\ell+k}.
  $$
Therefore all monomials $Q$ of degree $n_\ell$ contribute in $f_\lambda$ only for $v_k$ with $k=\nu n_\ell, \ldots, D n_\ell$, and this contribution is the coefficient
\begin{equation}\label{eq.vk}
 \sum_{I = 0}^{n_\ell}
     \lambda_{I,n_\ell-I}
        \sum_{*=k-\nu n_\ell} 
        \alpha_{i_1}\cdots\alpha_{i_I} a_{j_1}\cdots a_{j_{n_\ell-I}}.
\end{equation}

It follows that the monomials of $Q$ of degree $n_\ell$ contribute in $M$ for a block $M_{n_\ell}$ with $(n_\ell+1)$ columns and $(n_\ell (D-\nu)+1)$ lines, corresponding to the $(n_\ell+1)$ parameters $\lambda_{0,n_\ell}, \ldots, \lambda_{n_\ell, 0}$ appearing in $v_k(\lambda)$, for $k=\nu n_\ell, \ldots, D n_\ell$.

In this block $M_{n_\ell}$ the $(I+1)$-th column, for $I=0, \ldots, n_\ell$, starts at row $1$ with entries $$\displaystyle\sum_{*=0} 
        \alpha_{i_1}\cdots\alpha_{i_I} a_{j_1}\cdots a_{j_{n_\ell-I}},
        \displaystyle\sum_{*=1} 
        \alpha_{i_1}\cdots\alpha_{i_I} a_{j_1}\cdots a_{j_{n_\ell-I}}, \ \mathrm{etc.}, $$ 
        and ends at row $N_I+1$, with entry 
        $\displaystyle\sum_{*=N_I} 
        \alpha_{i_1}\cdots\alpha_{i_I} a_{j_1}\cdots a_{j_{n_\ell-I}}$.
        \begin{notn}\label{not.Lambda L}
            For $i\ge 0$, we denote by \(C_{\geq i}\) the data 
            $$
         C_{\geq i}=   (\lambda_i,\lambda_{i+1},\ldots,\lambda_{D_1-\nu_1},l_i,l_{i+1},\ldots,l_{D_2-\nu_2})\in \NN^{D_1+D_2-\nu_1-\nu_2-2i+2},
           $$
           and
     $$\Lambda=\lambda_1+\cdots+\lambda_{D_1-\nu_1}, \ \ 
     L=l_1+\cdots+l_{D_2-\nu_2},$$
      
       $$\Lambda_!=\lambda_1!\cdots \lambda_{D_1-\nu_1}!, \ \ L_!=l_1! \cdots l_{D_2-\nu_2}! .$$
\end{notn}
        
        \begin{rem}\label{rem.coeff multinomiaux}
            A monomial $\mathfrak{m}$, in $\alpha_0, \ldots, \alpha_{D_1-\nu_1}, a_0, \ldots, a_{D_2-\nu_2}$, in column \(I+1\in \llbracket1,n_\ell+1\rrbracket \) and at row \(r \in \llbracket1, n_\ell(D-\nu)+1\rrbracket\) in $M_{n_ \ell}$ is of form 
            \begin{equation}\label{eq.monomial}
                \alpha_0^{\lambda_0}\alpha_1^{\lambda_1}\cdots\alpha_{D_1-\nu_1}^{\lambda_{D_1-\nu_1}} a_0^{l_0} a_1^{l_1}\cdots a_{D_2-\nu_2}^{l_{D_2-\nu_2}},
            \end{equation}
        with 
        \begin{equation}
            \label{eq.sum lambdas}
        \lambda_0+\cdots + \lambda_{D_1-\nu_1}=I,  l_0+\cdots + l_{D_2-\nu_2}=n_\ell-I,
        \end{equation}
        and 
        \begin{equation}\label{eq.poids ligne r}
            0\lambda_0+\cdots + (D_1-\nu_1) \lambda_{D_1-\nu_1} +(\nu_1-\nu_2)I+0l_0+\cdots + (D_2-\nu_2)l_{D_2-\nu_2}=r-1.
            \end{equation}
          Notice that the data \(C_{\ge 0}\) fix the monomial $\mathfrak{m}$. However, from~\eqref{eq.sum lambdas}, the data \(C_{\ge 1}\) with the column \(I\) already fix $\mathfrak{m}$. 
         The coefficient in front of such a monomial $\mathfrak{m}$ in column \(I\) is 
            \begin{align*}
                \binom{I} {\lambda_0,\ldots,\lambda_{D_1-\nu_1}} \cdot \binom{n_\ell-I} {l_0,\ldots,l_{D_2-\nu_2}},
            \end{align*} where the notation $ \displaystyle \binom{i}{i_1,\ldots,i_p}$ stands for the multinomial  
            $\displaystyle \frac{i!}{i_1!\cdots i_p!}$.
            Therefore this coefficient is
            \begin{align}\label{eq.Polynome binomial}
                P_{C_{\ge 1},I} & = \frac{(\lambda_0+1)\cdots I}{\Lambda_!} \cdot \frac{(l_0+1)\cdots (n_{\ell}-I)}{L_!} \nonumber\\
                & =\frac{1}{\Lambda_! L_!} 
                \prod_{k=0}^{I-(\lambda_0+1)}(I-k)
                \prod_{k=0}^{n_\ell-I-(l_0+1)} (n_\ell-I-k).
            \end{align}
            
            Notice that, according to Notation \ref{not.Lambda L},
            we have  \(I-\lambda_0=\Lambda\) and \(n_\ell-I-l_0=L\). So for a monomial $\mathfrak{m}$ of given and fixed exponents $C_{\ge 1}$, the coefficient \(P_{C_{\ge 1},I}\) is a polynomial \(P_{C_{\ge 1}}\) in $I$ of degree
            \begin{align}
            \label{eq.deg Pm}
            \deg(P_{C_{\ge 1}}) 
            =\Lambda+L.
             \end{align}
             Furthermore, the leading coefficient of \(P_{C_{\ge 1}}\) is 
             \begin{equation}\label{eq.leading coefficient}
                 \frac{(-1)^L}{\Lambda_! L_!}.
             \end{equation}
            
        \end{rem}
        
        
        \begin{lem}\label{lem.Formule combinatoire}
            Let $c  \in 
            \NN$ and $T=\{t_1,\ldots,t_{c+1}\} \subseteq 
            \CC$, with $t_i<t_j$ for $i<j$.
            \begin{enumerate}
                \item 
            There exists a finite sequence \((K_{T,j})_{j\in T}\) depending only on $T$, such that for any polynomial $P \in \CC[X]$ with \(\deg(P) \le c\), we have
            \[
                \sum_{j\in T} P(j)K_{T,j} = \mathrm{coeff}_{X^c}(P).
            \]
            In particular when $ c > \deg(P)$ we have $\sum_{j\in T} P(j)K_{T,j} = 0$. 
            \item 
            On the other side, for $c\le \deg(P)$, we have 
            \begin{align}\label{eq.P degree tau with c+1 tests}
            \sum_{j\in T} P(j)K_{T,j} & = \mathrm{coeff}_{X^c}(P) 
            + \mathrm{coeff}_{X^{c+1}}(P)\sum_{j=1}^{c+1}t_j^{c+1}K_{T,t_j}  \\
            & +\cdots + \mathrm{coeff}_{X^{\deg(P)}}(P)\sum_{j=1}^{c+1}t_j^{\deg(P)} K_{T,t_j}.  \nonumber
            \end{align}
             \end{enumerate}
        \end{lem}

        \begin{proof}  
            Let \(V_T\) be the Vandermonde matrix related to \(T\), that is to say  
            \begin{align*}
                V_T = \begin{pmatrix}
                    1 & 1 & \cdots & 1 \\
                    t_1 & t_2 & \cdots & t_{c+1} \\
                    \vdots & \vdots & & \vdots \\
                    t_1^{c} & t_2^{c} & \cdots & t_{c+1}^{c}.
                \end{pmatrix}
            \end{align*}
            Let us denote by \(v_T \) its determinant. We have
            \[ v_T=\prod_{1\le i< j \le c+1}(t_j-t_i),\]
            Let us first prove $(1)$. For this let \(C\) be the row matrix of coefficients of \(P\), completed at its right with $c-\deg(P)$ zeros. 
            Notice that  
            \begin{align*}
                C \cdot V_T = \left( P(t_1), P(t_2), \ldots, P(t_{c+1}) \right).
            \end{align*}
            
            Let \(V^{-1}_{T,c+1}\) be the last column of the inverse matrix of \(V_T\).  We have
            \begin{align*}
                \text{coeff}_{X^{c}}(P) = C\cdot V_T \cdot V^{-1}_{T,c+1} =  \left( P(t_1), P(t_2), \ldots, P(t_{c+1})\right)\cdot V^{-1}_{T,c+1}.
            \end{align*}
            Consequently, it is sufficient to take for \(K_{T,t_p}\) the \(p^{\text{th}}\) element of \(V^{-1}_{T,c+1}\), which is by Cramer's formula 
            \begin{align}\label{eq.coeff K}
                K_{T,t_p} = \frac{(-1)^{c+p+1}}{v_T}\det(V_{T\setminus\{t_p\}}) = \frac{1}{\prod_{i\neq p} (t_p - t_i)}.
            \end{align}
            This define the sequence $K_{T,t_1}, \ldots, K_{T,t_{c+1}}$, and proves $(1)$. 
            
            To prove $(2)$, assuming $c\le \deg(P)$, we consider 
$\widetilde{V}_T$ the square matrix of size $\deg(P)+1$, defined as the matrix
$V_T$ first completed at its bottom by the matrix
$$
\begin{pmatrix}
t_1^{c+1} & \cdots & t_{c+1}^{c+1} \\
\vdots & & \vdots \\
t_1^{\deg(P)} & \cdots & t_{c+1}^{\deg(P)}
\end{pmatrix},            
$$
and then at its right by the matrix with $\deg(P)+1$ rows and $\deg(P)-c$ columns having only zeros for coeficients.             
Let us also consider  $\widetilde{V}^{-1}_{T,c+1}$ the last column of the inverse matrix of $V_T$, but this time completed at its bottom by $\deg(P)-c$ zeros, in order to obtain a column with $\deg(P)+1$ rows. With this notation we have
            $$
            \widetilde{V}_T\cdot \widetilde{V}_{T,c+1}^{-1}=
            \begin{pmatrix}
            0 \\
            \vdots \\
            0 \\
            1 \\
            \sum_{j=1}^{c+1}t_j^{c+1}K_{T,t_j} \\
            \vdots \\
            \sum_{j=1}^{c+1}t_j^{\deg(P)} K_{T,t_j}
            \end{pmatrix},
            $$
            where the unit appears at row $c+1$. Then the same computation as above, involving $\widetilde{V}_T$ and $\widetilde{V}^{-1}_{T,c+1}$ instead of $ V_T$ and $V^{-1}_{T,c+1}$, shows that 
            \begin{align*}\label{eq.P degree tau with c+1 tests}
            \sum_{j\in T} P(j)K_{T,j}= & \textrm{coeff}_{X^c}(P) 
            + \textrm{coeff}_{X^{c+1}}(P)\sum_{j=1}^{c+1}t_j^{c+1}K_{T,t_j}  \\
            & +\cdots + \textrm{coeff}_{X^{\deg(P)}}(P)\sum_{j=1}^{c+1}t_j^{\deg(P)} K_{T,t_j}.\qedhere
            \end{align*}
         \end{proof}

\begin{rem}\label{rem.formule binomiaux}
 In case $T=\{0,1, \ldots, c\}$ in Lemma \ref{lem.Formule combinatoire}, we have 
 $$K_{T,p}=\frac{(-1)^{c+p}}{c!}\binom{c}{p}.$$ 
 Therefore, for any polynomial $P$ with $\deg(P)\le c$, the following classical formula follows from Lemma \ref{lem.Formule combinatoire} $(1)$ 
 \begin{equation}\label{eq.formule binomiaux}
     \sum_{j=0}^{c}(-1)^j\binom{c}{j}P(j)=c!(-1)^c\mathrm{coeff}_{X^c}(P).
 \end{equation}
\end{rem}
We now generalize formula \eqref{eq.formule binomiaux} to the case of 
multivariate polynomials. 

\begin{lem}\label{lem.Formule combinatoire multivariee}
    Let \(P \in \mathbb{C}[X_1,\ldots,X_k]\) of total degree \(d\). Let \( (w_1,\ldots,w_k) \in \mathbb{N}^k \) such that \( d< \sum_{i=1}^k w_i  \). Then,
    \[
        \sum_{(i_1,\ldots,i_k) \in \llbracket0,w_1\rrbracket\times\cdots\times\llbracket 0,w_k\rrbracket} (-1)^{i_1+\cdots+i_k} \binom{w_1}{i_1} \cdots \binom{w_k}{i_k} P(i_1,\ldots,i_k) =0
    \]
\end{lem}
\begin{proof}
    By linearity it is sufficient to prove the formula in the case where \(P\) is a monomial \(X_1^{\gamma_1}\cdots X_k^{\gamma_k}\) of degree at most \(d\). Since \(d < \sum_{i=1}^k w_i\), it means there exists \(l \in \llbracket1,k\rrbracket\) such that \(\gamma_l<w_l\). Therefore
    \begin{align*}
        & \sum_{(i_1,\ldots,i_k) \in \llbracket 0,w_1\rrbracket\times\cdots\times\llbracket 0,w_k\rrbracket} (-1)^{i_1+\cdots+i_k} \binom{w_1}{i_1} \cdots \binom{w_k}{i_k} i_1^{\gamma_1}\cdots i_k^{\gamma_k}
        \\
        = & \prod_{j=1}^k \left( \sum_{i_j \in \llbracket0,w_j\rrbracket} (-1)^{i_j} \binom{w_j}{i_j} i_j^{\gamma_j} \right) =0,
    \end{align*} since the \(l^{\text{th}}\) factor of the product is zero by Remark~\ref{rem.formule binomiaux}.
\end{proof}

        \begin{notn}\label{notn.Matrice avec choix de colonnes}
        For $\tau \le n_\ell+1$ and $T=T_{\ell}=\{t_1,\ldots,t_\tau\}\subset \{1, \ldots, n_\ell+1\}$, with $t_1 < \cdots < t_\tau$,  we consider \(M_{n_\ell,T}\), the submatrix  of $M_{n_\ell}$ obtained from $M_{n_\ell}$ by only keeping the columns indexed by elements of \(T\). 
        And in the same way, we denote by $M_{T_{0}, \ldots, T_{\ell_d}}$ the submatrix of the complete Bautin matrix $M$ where we only keep the blocks \(M_{n_{\ell},T_{\ell}}\), $\ell=0, \ldots, \ell_d$, corresponding to some degrees $n_0, \ldots, n_{\ell_d}=d$, and for each of them, corresponding to a specific choice of columuns $T_{0}, \ldots, T_{\ell_d}$. 
        The matrix \(M_{n_\ell,T}\) has  
        $$r_T=\max_{1 \le i \le \tau}((n_\ell-(t_i-1))D_2+(t_i-1)D_1)- \nu n_\ell+1
        $$ rows and $\tau$ columns.

        With the notation of Lemma \ref{lem.Formule combinatoire} we denote by \(\mathcal{K}_T\) the following upper triangular matrix of size \(\tau\times \tau\)
        \begin{align*}
            (\mathcal{K}_T)_{i,j} = \begin{cases}
            \left(\frac{\alpha_0}{a_0}\right)^{t_j-t_i} \frac{K_{\{t_1,\ldots,t_j\},t_i}}{K_{\{t_1,\ldots,t_j\},t_j}}  & \text{ if }i\leq j\le \tau \\
            0 & \text{ otherwise.}
            \end{cases}
        \end{align*}
        Notice that all diagonal elements of the matrix $\mathcal{K}_T$ are equal to one.
        \end{notn}
        
        \begin{rem}\label{rem:multiplicites differentes}
        In the case $\nu_1\not=\nu_2$ the matrix \(M_{n_{\ell},T}\) is in column echelon form. Indeed, in this case, each column is shifted down by a strictly monotone number of zeros, from one column to the next one, and each shifted column has for upper coefficient a (nonzero) monomial in $a_0$ and $\alpha_0$. In particular the corresponding block $M_{ n_\ell, T}$ has maximal rank~$\tau$. 
        \end{rem}
        
        Following Remark \ref{rem:multiplicites differentes}, in what follows, we only deal with the case \(\nu_1=\nu_2=\nu\). 
        
        \begin{prop}\label{prop.triangulation}
        With the notation above 
        
        \begin{enumerate}
            \item the (upper minor of size $\tau$ of the) matrix \(M_{n_\ell,T} \cdot \mathcal{K}_T\) is lower triangular,
            \item the diagonal coefficients of  (the upper minor of size $\tau$ of) this matrix at row $r\in \llbracket 1, \tau\rrbracket$ is
            $$\frac{\prod_{i<r} (t_r-t_i) }{(r-1)!}  \alpha_0^{t_r-r}a_0^{n_\ell-(t_r+r-2)}(\alpha_1a_0-\alpha_0a_1)^{r-1} .$$
        \end{enumerate}
         
        \end{prop}
        
        \begin{proof}
       For \(r \in \llbracket 1,\tau \rrbracket \) and \(c\geq r\), we compute the coefficient of the matrix \(M_{n_\ell,T}\cdot\mathcal{K}_T\) at its row $r$ and its column $c$.
       Under the action of the multiplication by \(\mathcal{K}_T\), a monomial  \[\mathfrak{m}_i = \alpha_0^{\lambda_0-t_{c}+t_i}\alpha_1^{\lambda_1}\cdots\alpha_{D_1-\nu}^{\lambda_{D_1-\nu}} a_0^{l_0+t_{c}-t_i} a_1^{l_1}\cdots a_{D_2-\nu}^{l_{D_2-\nu}}\] from row $r$ and a column \(i \leq c\) of the matrix \(M_{n_\ell,T}\) is sent to the monomial 
       \[\mathfrak{m} = \alpha_0^{\lambda_0}\alpha_1^{\lambda_1}\cdots\alpha_{D_1-\nu}^{\lambda_{D_1-\nu}} a_0^{l_0} a_1^{l_1}\cdots a_{D_2-\nu}^{l_{D_2-\nu}}\] in the same row $r$ and the column \(c\) of the matrix \(M_{n_\ell,T}\cdot \mathcal{K}_T\). 
       Consequently, using the notation $C_{\ge 1}=(\lambda_1, \ldots, \lambda_{D_1-\nu}, l_1, \ldots,l_{D_2-\nu})$ and $P_{C_{\ge 1}}$ of Remark \ref{rem.coeff multinomiaux} we have
       \begin{align}\label{eq.coeff After Triang}
           \text{coeff}_{\mathfrak{m}}((M_{n_\ell,T}\cdot \mathcal{K}_T)_{r,c}) & = \sum_{i=1}^{c} \frac{K_{\{t_1,\ldots,t_{c}\},t_i}} {K_{\{t_1,\ldots,t_{c}\},t_{c}}} \text{coeff}_{\mathfrak{m}_i}((M_{n_\ell,T})_{r,i}) \nonumber \\
           & = \frac{1} {K_{\{t_1,\ldots,t_{c}\},t_{c}}}\sum_{i=1}^{c} K_{\{t_1,\ldots,t_{c}\},t_i} P_{C_{\ge 1}}(t_{i}-1).
       \end{align}
    where, by~\eqref{eq.deg Pm} \(P_{C_{\ge 1}}\) is a polynomial of degree \(\sum_{i\geq 1} \lambda_i+l_i\). By~\eqref{eq.poids ligne r} \(\deg(P_{C_{\ge 1}}) \leq r-1 \le c-1\), and \(\deg(P_{C_{\ge 1}})<r-1\) in case \(\sum_{i\geq 2} l_i+\lambda_i \neq 0\). 
      By Lemma~\ref{lem.Formule combinatoire} $(1)$, we deduce from~\eqref{eq.coeff After Triang} that
      \[
        \text{coeff}_{\mathfrak{m}}((M_{n_\ell,T}\cdot \mathcal{K}_T)_{r,c}) = \frac{1} {K_{\{t_1,\ldots,t_{c}\},t_{c}}} \text{coeff}_{X^{c-1}}(P_{C_{\ge 1}}(X-1)).
      \]
By \eqref{eq.leading coefficient},  \(\text{coeff}_{\mathfrak{m}}((M_{n_\ell,T}\cdot \mathcal{K}_T)_{r,c})\) equals 
      \begin{equation*}
          \frac{(-1)^{l_1}}{\lambda_1! l_1! K_{\{t_1,\ldots,t_r\},t_{r}}}
      \end{equation*}
      in case \(c=r=\deg(P_{C_{\ge 1}})+1\), and zero otherwise. This proves in particular statement $(1)$, and shows moreover that for the diagonal coefficients of the matrix \(M_{n_\ell,T}\cdot \mathcal{K}_T\), where $c=r$, only monomials of the form \(\alpha_0^{\lambda_0}\alpha_1^{\lambda_1}a_0^{l_0}a_1^{l_1}\) occur,  each one with coefficient 
      \begin{equation*}
          \frac{(-1)^{l_1}}{(\lambda_1+l_1)! K_{\{t_1,\ldots,t_{r}\},t_{r}}} \binom{\lambda_1+l_1}{l_1}.
      \end{equation*}
    Considering that  for monomials where \(\sum_{i\geq 2} \lambda_i+l_i =0\) we have from~\eqref{eq.sum lambdas} and \eqref{eq.poids ligne r} that $\lambda_0+\lambda_1=t_r-1$, $\lambda_1+l_1=r-1$ and \(l_0+l_1=n_\ell-t_r+1\), the  diagonal coefficient at row $r$ of \(M_{n_\ell,T}\cdot \mathcal{K}_T\) is
    \begin{align*}
        & \sum_{l_1 = 0}^{r-1} \frac{(-1)^{l_1}}{(r-1)! K_{\{t_1,\ldots,t_r\},t_r}} \binom{r-1}{l_1} \alpha_0^{\lambda_0}\alpha_1^{\lambda_1}a_0^{l_0}a_1^{l_1} \\
        = & \frac{1}{(r-1)! K_{\{t_1,\ldots,t_r\},t_r}} \alpha_0^{t_r-r}a_0^{n_\ell-(t_r+r-2)} \sum_{l_1=0}^{r-1} (-1)^{l_1} \binom{r-1 }{l_1} \alpha_0^{l_1}\alpha_1^{r-1-l_1}a_0^{r-1-l_1}a_1^{l_1} \\
        = & \frac{1}{(r-1)! K_{\{t_1,\ldots,t_r\},t_r}}  \alpha_0^{t_r-r}a_0^{n_\ell-(t_r+r-2)}(\alpha_1a_0-\alpha_0a_1)^{r-1} .
    \end{align*}
    This completes the proof of statement $(2)$. 
    \end{proof}
        
\begin{rem}
    Note that in case $T=\{1, \ldots, n_\ell+1\}$ the diagonal coefficient
    at row $r$ of the (upper minor of size $n_\ell+1$ of the) triangular matrix \(M_{n_\ell,T}\cdot \mathcal{K}_T\) is $a_0^{n_\ell-2(r-1)}(\alpha_1a_0-\alpha_0a_1)^{r-1}$, since in this case $t_r=r$ and $K_{\{t_1,\ldots,t_r\},t_r}=1/(r-1)!$. 
\end{rem}

In Proposition \ref{prop.triangulation} above we proved that the matrix $M_{n_\ell, T}\cdot \mathcal{K}_T$ is lower triangular and we computed its diagonal coefficients. Let us now compute the other coefficients of this matrix. For this, in Proposition \ref{prop.triangulation complete} below, we first compute the coefficient of a monomial 
 $\mathfrak{m} = \alpha_0^{\lambda_0}\alpha_1^{\lambda_1}\cdots\alpha_{D_1-\nu}^{\lambda_{D_1-\nu}} a_0^{l_0} a_1^{l_1}\cdots a_{D_2-\nu}^{l_{D_2-\nu}}$
 appearing  at row $r\in \llbracket 1, \tau \rrbracket$ and  column $c \le r-1$ in $M_{n_\ell, T}\cdot \mathcal{K}_T$. 
 Notice that this triangulation provides monomials with negative exponent for \(a_0\), since in $\mathcal{K}_T$ denominators $a_0^{t_j-t_i}$, for $j > i$, appear. So from now, we will work with monomials where we formally allow exponent \(l_0\) to be negative, that is to say, we consider monomials appearing in $M_{n_\ell, T}\cdot \mathcal{K}_T$ as elements of the Laurent ring \(\mathbb{C}[\alpha_0,\alpha_1,\ldots,\alpha_{D_1-\nu},
 a_0^{-1},a_0, a_1,\ldots,a_{D_2-\nu}]\).
 
 Also notice that equations \eqref{eq.sum lambdas} and \eqref{eq.poids ligne r} are satisfied for monomials appearing  
  in $M_{n_\ell, T}$, as well as for monomials virtually appearing in $M_{n_\ell, T}\cdot \mathcal{K}_T$, since multiplication by factors $\left(\dfrac{\alpha_0}{a_0}\right)^{t_j-t_i}$, $j\ge i$,  occurring in $\mathcal{K}_T$ preserves those equations.

\begin{prop}\label{prop.triangulation complete}
With the notation above, at row $r\in \llbracket 1, r_T \rrbracket$ and column $c\le \min(r-1,\tau)$ in $M_{n_\ell, T}\cdot \mathcal{K}_T$, for any exponents 
\[
    (\lambda_0,\lambda_1,\lambda_2,\ldots,\lambda_{D_1-\nu},l_0,l_1,\ldots,l_{D_2-\nu}) \in \mathbb{N}^{D_1-\nu+1}\times\mathbb{Z}\times\mathbb{N}^{D_2-\nu}
\] satisfying~\eqref{eq.poids ligne r}, and satisfying \eqref{eq.sum lambdas}, that is to say
\begin{align*}
    \lambda_0+\Lambda = t_c-1, \  l_0+L = n_{\ell}+1 -t_{c},
\end{align*}
the coefficient of the monomial  
$$\mathfrak{m} = \alpha_0^{\lambda_0}\alpha_1^{\lambda_1}\cdots\alpha_{D_1-\nu}^{\lambda_{D_1-\nu}} a_0^{l_0} a_1^{l_1}\cdots a_{D_2-\nu}^{l_{D_2-\nu}}$$ is $0$ in case $c\ge \Lambda+L+2$, and in case $c\le \Lambda+L+1$, this coefficient is  
\begin{equation}\label{eq.produit sigma ki}
 \sigma_{\Lambda+L-c+1}\cdot k_{c-1}+   \sigma_{\Lambda+L-c}\cdot k_c+\cdots + \sigma_{0} \cdot k_{\Lambda+L}
\end{equation}
where  
$$ k_{c-1}=\frac{1}{K_{\{t_1,\ldots,t_{c}\},t_{c}}}$$ 
and for $p=c,\ldots , \Lambda+L,$
$$ k_p= \frac{\sum_{j=1}^c t_j^p
            K_{\{t_1,\ldots,t_{c}\},t_j}} {K_{\{t_1,\ldots,t_{c}\},t_{c}}},$$  
            depend only on $c$ and of the choice of $T$, and where 
             \begin{equation*}
            \sigma_{\Lambda+L-p}=\frac{(-1)^{\Lambda-p}}{\Lambda_!L_!}\sigma_{\Lambda+L-p}(1,2, \ldots,\Lambda,n_\ell+2-L, \ldots,n_\ell+1),
            \end{equation*}
            with
            $\sigma_k(x_1, \ldots,x_{\Lambda+L})=\displaystyle \sum_{1\le i_1< \ldots< i_k \le \Lambda+L}x_{i_1}\cdots x_{i_k}$ be the $k$-th elementary symmetric polynomial.
            \end{prop}

\begin{proof}
       We fix a row $r\in \{1, \ldots, \tau\}$, a column $c\le r-1$, and a monomial \(\mathfrak{m} = \alpha_0^{\lambda_0}\cdots \alpha_{D_1-\nu}^{\lambda_{D_1-\nu}} a_0^{l_0} \cdots a_{D_2-\nu}^{l_{D_2-\nu}}\) in $M_{n_\ell, T}\cdot \mathcal{K}_T$ verifying the hypotheses of the proposition.
              
       As already noticed at \eqref{eq.coeff After Triang} in the proof of Proposition \ref{prop.triangulation}, this coefficient comes from the coefficients of monomials $\mathfrak{m}_i$ in \(M_{n_\ell, T}\) which appear at the same row $r$ as $\mathfrak{m}$, at columns $i=\{1, \ldots, c \}$, with the same exponents $\lambda_1,\cdots, \lambda_{D_1-\nu},  l_1,\cdots,l_{D_2-\nu}$ as in $\mathfrak{m}$, but with exponent $\lambda_0-t_c+t_i$ for $\alpha_0$, and exponent $l_0+t_c-t_i $ for $a_0$. Notice that this monomial exists in \(M_{n_\ell, T}\) if and only if \(\lambda_0-t_c+t_i \ge 0\) and \(l_0+t_c-t_i \ge 0\).
       This gives rise to the following expression for the coefficient of $\mathfrak{m}$
       \begin{equation*}
           \text{coeff}_{\mathfrak{m}}= \frac{1} {K_{\{t_1,\ldots,t_{c}\},t_{c}}}
           \sum_{\substack{i\in \llbracket 1,c \rrbracket \text{ st } \\ \lambda_0-t_c+t_i \ge 0 \text{ and } l_0+t_c-t_i  \ge 0}} K_{\{t_1,\ldots,t_{c}\},t_i} P_{C_{\ge 1}}(t_i - 1)
        \end{equation*}
        with $P_{C_{\ge1}}$ the polynomial of degree $\Lambda + L$ defined in Remark \ref{rem.coeff multinomiaux}. 
        Since \(P_{C_{\ge 1}}\) is defined by the expression~\eqref{eq.Polynome binomial}, this polynomial cancels 
       at $0, \ldots, \Lambda-1$ and $n_\ell-L+1, \ldots, n_\ell$.
       Moreover assuming \(\lambda_0-t_c+t_i <0\), it implies that $t_i - 1 \le t_c - \lambda_0 - 2 = \Lambda -1$. Consequently \(P_{C_{\ge 1}}(t_i - 1) = 0\). Similarly, if \(l_0+t_c-t_i < 0\), then $t_i - 1 \ge l_0 + t_c = n_{\ell} - L +1$ which again implies \(P_{C_{\ge 1}}(t_i - 1) = 0\). Consequently, the sum above is unchanged if we sum over all \(\llbracket 1,c \rrbracket\):
 \begin{equation}\label{eq.bis coeff After Triang}
           \text{coeff}_{\mathfrak{m}} = \frac{1} {K_{\{t_1,\ldots,t_{c}\},t_{c}}}
           \sum_{i=1}^{c} K_{\{t_1,\ldots,t_{c}\},t_i} P_{C_{\ge 1}}(t_i - 1). 
       \end{equation}
       
       \smallskip
       \begin{itemize}
           \item[-] 
       In case $c\ge \Lambda + L+2$, by Lemma \ref{lem.Formule combinatoire} $(1)$,  $\mathrm{coeff}_\mathfrak{m}=0$. 
       
       \smallskip
       \item[-] In case $c\le \Lambda + L+1$, by 
      \eqref{eq.P degree tau with c+1 tests} of Lemma \ref{lem.Formule combinatoire} $(2)$, we have 
        \begin{align*}
            \mathrm{coeff}_\mathfrak{m}
            = \frac{1} {K_{\{t_1,\ldots,t_{c}\},t_{c}}} 
            \Big( &
              \mathrm{coeff}_{X^{c-1}}(P_{C_{\ge 1}}(X-1))  \\
            +\  & \mathrm{coeff}_{X^c}(P_{C_{\ge 1}}(X-1))\sum_{j=1}^c t_j^c
            K_{\{t_1,\ldots,t_{c}\},t_j} \\
            +\  & \cdots + \mathrm{coeff}_{X^{\Lambda+L}}(P_{C_{\ge 1}}(X-1))\sum_{j=1}^c t_j^{\Lambda+L} K_{\{t_1,\ldots,t_{c}\},t_j} \Big).
            \end{align*}
       We recall from \eqref{eq.Polynome binomial} that
       \begin{align*}\label{eq.bis Polynome binomial}
                P_{C_{\ge 1}}(X-1) =\frac{(-1)^L}{\Lambda_! L_!} 
                \prod_{k=1}^{\Lambda}(X-k)
                \prod_{k=1}^{L} (X-(n_\ell+2-k)).
            \end{align*}
            It follows that for $p=c-1, \ldots, \Lambda+L$
            \begin{equation}\label{eq.coefficient degre p de P}
            \mathrm{coeff}_{X^p}(P_{C_{\ge 1}}(X-1))=
            \frac{(-1)^{\Lambda-p}}{\Lambda_! L_!}
            \sigma_{\Lambda+L-p}(1,2, \ldots,\Lambda,n_\ell+2-L, \ldots,n_\ell,n_\ell+1). 
            \end{equation}
            \end{itemize}
\end{proof}

\begin{lem}\label{lem.sigma2 polynome a valeurs entieres}
Let $N,R,p\in \NN$ with $p \le R< N$ and 
$f_p^{R,N}:\llbracket 0,R \rrbracket \to \NN$ be the function defined by 
$f_p^{R,N}(n)=\sigma_p(1,2, \ldots, R-n, N-n, \ldots, N)$, with the notation of Proposition \ref{prop.triangulation complete}. There exists
a polynomial $P_p^{R,N}$ of degree $p$ such that $f_p^{R,N}(n)=P_p^{R,N}(n)$, for all $n \in \llbracket 0,R \rrbracket$. 
\end{lem}

\begin{proof} The proof is by induction on $p \ge 0$. Notice first 
that $f_0^{R,N} \equiv 1$ is a constant function. 
 Then for $p\ge 1$ a direct computation shows that for all $n\in \llbracket 0,R-1 \rrbracket$ 
 $$
 f_p^{R,N}(n+1)-f_p^{R,N}(n)=(N-R-1)\sigma_{p-1}(1, 2, \ldots, R-1-n,N-n, \ldots, N).
 $$
 Since by induction hypothesis the right hand side of the above equation equals the polynomial $(N-R-1)P_{p-1}^{R-1,N}$ of degree $p-1$ on $\llbracket 0,R-1 \rrbracket$, so does $f_p^{R,N}(n+1)-f_p^{R,N}(n)$. 
 
 We just need to show this implies there exists a polynomial \(P_p^{R,N}\) which equals \(f_p^{R,N}\) on $\llbracket 0,R \rrbracket$. To do that, we follow for instance \cite[Proposition I.7.3 (b)]{Hart}. Let us consider the family of Newton polynomials \((\binom{X}{r})_{r\in \NN}\) defined by \(\binom{X}{r} = \frac{1}{r!}\prod_{i=0}^{r-1}(X-i)\). We can easily notice that \((\binom{X}{r})_{r \le d}\) is a basis of polynomials of degree at most \(d\). Moreover, a simple calculation shows that for \(r>0\), \(\binom{X+1}{r}-\binom{X}{r} = \binom{X}{r-1}\).
 Writing $(N-R-1)P_{p-1}^{R-1,N}$ in this basis, we obtain the coefficients \((c_{p-1},\ldots,c_0)\)
 \begin{equation*}
     (N-R-1)P_{p-1}^{R-1,N}(X) = c_{p-1}\binom{X}{p-1}+c_{p-2}\binom{X}{p-2}+\cdots+c_0.
 \end{equation*}
 Let us define
 \begin{equation*}
     P_{p}^{R,N}(X) = c_{p-1}\binom{X}{p}+c_{p-2}\binom{X}{p-1}+\cdots+c_0\binom{X}{1} + K
 \end{equation*}
 where the constant \(K\) is chosen to get  \(P_{p}^{R,N}(R) =  f_{p}^{R,N}(R)\).
 It follows that for any $n\in \llbracket 0,R-1 \rrbracket$, \(P_{p}^{R,N}(n+1) -  P_{p}^{R,N}(n) = f_{p}^{R,N}(n+1) - f_{p}^{R,N}(n)\). An easy induction finishes the proof.
\end{proof}


\section{Triangulation of Bautin blocks - The singular case}\label{section 3}
In this section we continue to focus on the case \(\nu=\nu_1=\nu_2\). We assume that the upper $\tau\times \tau$ minor of the lower triangular matrix $ M_{n_\ell,T}\cdot \mathcal{K}_T$ has determinant zero, where $T=\{t_1, \ldots, t_\tau\}\subset \{1, \ldots, n_\ell+1\}$, with $t_1 < \cdots < t_\tau$,  is a choice of $\tau$ columns in the Bautin block $M_{n_\ell}$, according to Notation \ref{notn.Matrice avec choix de colonnes}.
By Proposition \ref{prop.triangulation}, the condition that the determinant of the upper \(\tau\times\tau\) minor of $M_{n_\ell,T}\cdot \mathcal{K}_T$ is zero  means that $\alpha_0a_1 = \alpha_1a_0$. 
There exists $\mu\in \CC\setminus\{0\}$ such that 
$\alpha_0=\mu a_0, \alpha_1=\mu a_1$,  
and we denote by \(k\) the largest integer in \(\llbracket 0,\min(D_1,D_2)- \nu \rrbracket \) satisfying 
\begin{equation}\label{eq.singular condition}
 (\alpha_0,\ldots,\alpha_k)=\mu\cdot (a_0,\ldots,a_k),
\end{equation}
assuming in this section that $k\ge 1$. 
In the lower triangular matrix $M_{n_\ell,T}\cdot \mathcal{K}_T$ we fix a row $r\in \{1, \ldots, r_T\}$ and a column $c\le r-1$. 
The entry in $M_{n_\ell,T}\cdot \mathcal{K}_T$ at row $r$ and column $c$ is a polynomial  in 
       $\alpha_0, \ldots, \alpha_{D_1-\nu}$,$a_0,\ldots, $ $a_{D_2-\nu}$ (actually a Laurent polynomial, since $a_0$ may have negative power $l_0$ after multiplication of $M_{n_\ell,T}$ by $\mathcal{K}_T$), which is a linear combination of monomials of the form
        \[\mathfrak{m} = \alpha_0^{\lambda_0}\alpha_1^{\lambda_1}\cdots\alpha_{D_1-\nu}^{\lambda_{D_1-\nu}} a_0^{l_0} a_1^{l_1}\cdots a_{D_2-\nu}^{l_{D_2-\nu}}\]
               with, according to \eqref{eq.sum lambdas},
        \begin{equation}\label{eq.tris bis sum lambdas}
        \lambda_0+\Lambda=t_c-1, \ \  l_0+L=n_\ell-t_c+1.
        \end{equation}
        
        If, according to condition \eqref{eq.singular condition}, we replace for \(i \in \llbracket 0,k\rrbracket\), $\alpha_i$ by $\mu a_i$ in the monomial $\mathfrak{m}$ we get a monomial of the form 
         \begin{equation}\label{eq.tris m tilde}
         \widetilde{\mathfrak{m}} =\mu^{\lambda_0+\ldots+\lambda_k}\cdot a_0^{\lambda_0+l_0} \cdots a_k^{\lambda_k+l_k}\cdot  \alpha_{k+1}^{\lambda_{k+1}}\cdots\alpha_{D_1-\nu}^{\lambda_{D_1-\nu}} a_{k+1}^{l_{k+1}}  \cdots a_{D_2-\nu}^{l_{D_2-\nu}}.
         \end{equation}
     
        \begin{notn}\label{not.tris C Lambda L W}
         For $i=0, \ldots, D-\nu$, let us denote
        $$\Lambda_{\ge i}=\lambda_i + \ldots + \lambda_{D_1-\nu}, \ \ L_{\ge i}=l_i + \ldots + l_{D_2-\nu},$$
        $$W_{\ge i}= i\lambda_i + \ldots + (D_1-\nu)\lambda_{D_1-\nu}+ il_i + \ldots + (D_2-\nu) l_{D_2-\nu}.$$
        Notice that up to now we have used $\Lambda$ instead of $\Lambda_{\ge 1}$ and $L$ instead of $L_{\ge 1}$ (see Notation \ref{not.Lambda L}); In what follows we use those two notations.
        \end{notn}

        \begin{rem}\label{rem.tris m tilde fixé par C}
          Using Notation \ref{not.Lambda L}, in case $C_{\ge k+1}$ is fixed, so is  $\lambda_0+\ldots+\lambda_k$, since 
           \begin{equation}\label{eq.tris lambda0 lambda1}
         \lambda_0+\ldots+\lambda_k=t_c-1-\Lambda_{\ge k+1}.
        \end{equation}

          On the other hand, since 
          \(\lambda_0+l_0 = n_{\ell}-L_{\geq 1}-\Lambda_{\geq 1}\),
          monomials $\mathfrak{m}$ contribute, after the substitution of $\alpha_i$ by $\mu a_i$ for \(i \in \llbracket0,k\rrbracket\) in $\mathfrak{m}$, to a monomial $\widetilde{\mathfrak{m}}$ defined only by the data $T$, $c$, $(w_1,\ldots,w_k)$, and $C_{\ge k+1}$, where for all \(i \geq 0\), \(w_i=\lambda_i+l_i\), and according to~\eqref{eq.poids ligne r}, the sequence
          $(w_1,\ldots,w_k)$ verifies 
          $$
         w_1+2w_2+\cdots+kw_k = r-1-W_{\geq k+1}.
          $$ 
          \end{rem}
          
          Also notice 
          that  
          $$
          w_0 = \lambda_0+l_0= n_{\ell} -w_1-\cdots-w_k-\Lambda_{\geq k+1}-L_{\ge k+1}.$$

          \begin{rem}\label{rem.positivity}
              By Proposition \ref{prop.triangulation complete},  in case $\Lambda+L-c+1<0$, the coefficient of $\mathfrak{m}$ is zero, so we may assume that $\Lambda+L-c+1\ge 0$.
          \end{rem}

        \begin{rem}\label{rem.m tilde}
         By Proposition \ref{prop.triangulation} $(2)$ the entry at row $1$ and column $1$ of $M_{n_\ell,T}\cdot \mathcal{K}_T$
         is 
         \begin{equation}\label{eq.coefficient en haut a gauche} 
         \alpha_0^{t_1-1}a_0^{n_\ell-(t_1-1)}=\mu^{t_1-1}a_0^{n_\ell},
         \end{equation}
        and by \eqref{eq.tris m tilde}
        a monomial $\widetilde{\mathfrak{m}}$ at row $r$ and column $c$ in $ M_{n_\ell,T}\cdot \mathcal{K}_T$ is of form 
        \begin{equation}\label{eq.tris m tilde quand c=r-1}
            \widetilde{\mathfrak{m}}=\mu^{\lambda_0+\ldots+\lambda_k}  a_0^{w_0} \cdots a_k^{w_k} \alpha_{k+1}^{\lambda_{k+1}} \cdots \alpha_{D_1-\nu}^{\lambda_{D_1-\nu}} a_{k+1}^{l_{k+1}} \cdots a_{D_2-\nu}^{l_{D_2-\nu}}, 
        \end{equation}
        coming, after the use of \eqref{eq.singular condition}, from monomials (at same row $r$ and same column $c$) $\mathfrak{m}=\alpha_0^{\lambda_0} \cdots \alpha_{D_1-\nu}^{\lambda_{D_1-\nu}} a_0^{l_0} \cdots a_{D_2-\nu}^{l_{D_2-\nu}}$, where \(w_i = \lambda_i+l_i\), and where 
        \begin{equation}\label{eq.tris encadrement l1}
            (l_1,\ldots,l_k) \in \llbracket0,w_1\rrbracket\times\cdots\times\llbracket0,w_k\rrbracket.
        \end{equation}
    \end{rem}

\begin{prop}\label{prop.sous-diagonales nulles}
With the notation above, the entry at column $c$ and row 
$r\in \llbracket 1, (k+1)c-k \rrbracket $ 
in the matrix $M_{n_\ell,T}\cdot \mathcal{K}_T$  is zero for
$1\le r<  (k+1)c-k$,  and for $r=(k+1)c-k$ this entry is the nonzero constant
\begin{equation}\label{eq.coefficients jets egaux a ordre k}
\frac{\mu^{t_c-c}}{(c-1)! K_{\{t_1,\ldots,t_c\},t_c}} a_0^{n_{\ell}-c+1} (\alpha_{k+1}-\mu a_{k+1})^{c-1}.
\end{equation}
\end{prop}

\begin{proof}
 A monomial $\widetilde{\mathfrak{m}}$ as in \eqref{eq.tris m tilde quand c=r-1} of
      Remark \ref{rem.m tilde} depends only on the choice of $w_1,\ldots,w_k$ and $C_{\ge {k+1}}$, as pointed out in Remark \ref{rem.tris m tilde fixé par C}. So we fix a choice of data $w_1,\ldots,w_k,C_{\ge k+1}$ with \(w_1+2w_2+\cdots+kw_k=r-1-W_{\ge k+1}\). 
      On the other hand, the monomial $\widetilde{\mathfrak{m}}$ comes from all monomials $\mathfrak{m}$ as in Remark \ref{rem.m tilde}, parameterized by $(l_1,\ldots,l_k)\in\llbracket0,w_1\rrbracket\times\cdots\times\llbracket0,w_k\rrbracket$, and Proposition \ref{prop.triangulation complete} gives the coefficient $\mathrm{coeff}_\mathfrak{m}$ of $\mathfrak{m}$ at row $r$ and column $c$ in  $M_{n_\ell,T}\cdot \mathcal{K}_T$.

      Two cases may then happen: 
      \begin{enumerate}
        \item We first consider the case \(w_1=\cdots=w_k=0\), that is to say that the monomial \(\widetilde{\mathfrak{m}}\) is given by a monomial $\mathfrak{m}$, where \(\lambda_1=\cdots=\lambda_k=l_1=\cdots=l_k=0\). It implies in particular that
        \begin{equation}\label{eq.first inequality}
            \Lambda+L =\Lambda_{\ge k+1} + L_{\ge k+1} \le W_{\ge k+1}/(k+1) =(r-1)/(k+1).
            \end{equation} So, since $c\ge (r+k)/(k+1)$,
            \begin{equation}\label{eq.second inequality}
            0 \le \Lambda+L-c+1 \le \frac{r-1}{k+1} -\frac{r+k}{k+1}+1=0.
            \end{equation}
            Consequently inequalities \eqref{eq.first inequality} and \eqref{eq.second inequality} are equalities. In particular \(\lambda_{k+1}+l_{k+1}=c-1\) and for all \(i\ge 1\), if \(i \neq k+1\) we have \(\lambda_i=l_i=0\). Moreover there is a unique monomial \(\mathfrak{m}\) which gives \(\widetilde{\mathfrak{m}}\). Its coefficient is
            by \eqref{eq.produit sigma ki} of Proposition \ref{prop.triangulation complete}
        \[
            \mu^{\lambda_0} \sigma_0k_{c-1} = \mu^{t_c-1-\lambda_{k+1}}\frac{(-1)^{l_{k+1}}}{l_{k+1}!\lambda_{k+1}!K_{\{t_1,\ldots,t_c\},t_c}}.
        \]
        There are exactly \(c\) such monomials \(\widetilde{\mathfrak{m}}\) (we can choose the values of \(\lambda_{k+1}\) and \(l_{k+1}\)). The sum of these monomials is
        \begin{align*}
            & \sum_{l_{k+1}=0}^{c-1} \mu^{t_c-c+l_{k+1}}\frac{(-1)^{l_{k+1}}}{l_{k+1}!(c-1-l_{k+1})!K_{\{t_1,\ldots,t_c\},t_c}} {a_0}^{n_{\ell}-c+1} \alpha_{k+1}^{c-1-l_{k+1}} {a_{k+1}}^{l_{k+1}}
            \\
            = & \frac{\mu^{t_c-c}}{(c-1)! K_{\{t_1,\ldots,t_c\},t_c}} a_0^{n_{\ell}-c+1} \sum_{l_{k+1}=0}^{c-1} \binom{c-1}{l_{k+1}} \alpha_{k+1}^{c-1-l_{k+1}} (-\mu a_{k+1})^{l_{k+1}} \\
            = & \frac{\mu^{t_c-c}}{(c-1)! K_{\{t_1,\ldots,t_c\},t_c}} a_0^{n_{\ell}-c+1} (\alpha_{k+1}-\mu a_{k+1})^{c-1}.
        \end{align*}
      
        \item Otherwise, some \(w_i\) is not zero where \(1 \le i \le k\). By fixing the monomial \(\widetilde{\mathfrak{m}}\), we already fixed \(C_{\ge k+1}\), while \((l_1,\ldots,l_k)\in\llbracket0,w_1\rrbracket\times\cdots\times\llbracket0,w_k\rrbracket\). The coefficient of a monomial \(\mathfrak{m}\) which gives $\widetilde{\mathfrak{m}}$ is given by \eqref{eq.produit sigma ki} of Proposition~\ref{prop.triangulation complete}, and is
        \[
            \mathrm{coeff}_\mathfrak{m} = \sigma_{\Lambda+L-c+1}k_{c-1} + \cdots + \sigma_0 k_{\Lambda+L}
        \]
        where 
        \[
            k_p = \frac{K_p}{K_{\{t_1,\ldots,t_c\},t_c}}
        \]
        and 
        \[
            \sigma_p = \frac{(-1)^{p+L}}{\Lambda_!L_!}\sigma_p(1,2,\ldots,\Lambda,n_{\ell}+2-L,\ldots,n_{\ell}+1)
        \]
        with 
        \[
            K_{c-1}=1 \text{ \rm and }  K_{p} = \sum_{j=1}^c t_j^p K_{\{t_1,\ldots,t_c\},t_j} \ \mathrm{ when }\ p\ge c.
        \]
        
        Since \((l_1,\ldots,l_k)\) is going through the box \(\mathcal{B} = \llbracket0,w_1\rrbracket\times\cdots\times\llbracket0,w_k\rrbracket\), $\mathrm{coeff}_{\widetilde{\mathfrak{m}}}$ is
        \begin{align}\label{eq.Somme sur les l}
         \widetilde{K}  \sum_{l\in \mathcal{B}} (-1)^{l_1+\cdots+l_k} & \binom{w_1}{l_1} \cdots \binom{w_k}{l_k} \sum_{p=0}^{\Lambda+L-c+1} (-1)^p \\
         \nonumber
         & \times K_{\Lambda+L-p}\sigma_p(1,\ldots,\Lambda,n_{\ell}+2-L,\ldots,n_{\ell}+1),
        \end{align}
        with 
        $$ 
         \widetilde{K}=\frac{(-1)^{L_{\ge k+1}}}{
            \Lambda_{\ge k+1!} L_{\ge k+1!} (w_1!)\cdots(w_k!) K_{\{t_1,\ldots,t_c\},t_c}}.
        $$
        By Lemma~\ref{lem.sigma2 polynome a valeurs entieres}, for any \(i\), \(\sigma_i(1,\ldots,\Lambda,n_{\ell}+2-L,\ldots,n_{\ell}+1)\) is a polynomial in \(l_1+\cdots+l_k\) of degree at most \(i\). So, 
        \[\sum_{p=0}^{\Lambda+L-c+1} (-1)^pK_{\Lambda+L-p}\sigma_p(1,\ldots,\Lambda,n_{\ell}+2-L,\ldots,n_{\ell}+1)\]
        is a \(k\)-variate polynomial in \((l_1,\ldots,l_k)\) of total degree at most \(\Lambda+L-c+1\). 
        
        Finally, we have 
        \begin{align*}
            (k+1)(c-1) & \geq r-1 = W_{\ge 0} \\
            & \geq w_1+2w_2+\cdots+kw_k + (k+1)(\Lambda_{\ge k+1}+L_{\ge k+1}).
        \end{align*}
        It implies
        \[
            \Lambda+L -(c-1) \le \frac{kw_1+(k-1)w_2+\cdots+w_k}{k+1} < w_1+\cdots+w_k
        \]
        since we assumed that one of the \(w_i\) is nonzero.
        By Lemma~\ref{lem.Formule combinatoire multivariee}, the expression~\eqref{eq.Somme sur les l} above cancels.\qedhere
        \end{enumerate}
\end{proof}

\begin{rem}\label{rem.proportionalite}
In case $P_1=\mu P_2$, all columns of \(M_{n_{\ell},T}\) are linearly dependent. In the first row, only the entry of the first column is nonzero. Therefore, all columns but the first one in $M_{n_\ell,T}\cdot \mathcal{K}_T$ have zero for entries.  
\end{rem}

\section{Lacunary curves}\label{section 4}

From now on, we consider that the polynomial  $Q_\lambda=\sum_{i=1}^m \lambda_iQ_i$ is lacunary, in the sense that the family of parameters $\lambda=(\lambda_1, \ldots, \lambda_m)$ has zeros at prescribed places, and of course we then omit those zero parameters. 
We still denote by $d$ the degree of $Q_\lambda$, and by $m(=m_d\le d+1)$ the number of (nonzero) parameters of the family $\lambda$. 
We still denote by 
$n_0 < n_1 < \cdots < n_\ell< \ldots < n_{\ell_d}=d$ the sequence of $\ell_d$ degrees of the monomials $Q_i$ appearing in $Q_\lambda$. In addition, we return to the general case \(\nu_1 \ge \nu_2 = \nu\).

We have already considered in the previous sections some general choice of coefficients appearing in $Q_\lambda$, by selecting, for $\ell\in \llbracket 0,\ell_d\rrbracket$,  $\tau_\ell$
columns $T_\ell=\{t_{\ell,1}, \ldots, t_{\ell,\tau_\ell}\}\subset \{ 1, \ldots, n_\ell+1\}$, with $t_{\ell,1}< \cdots < t_{\ell,\tau_\ell}$, in the block $M_{n_\ell}$, corresponding to the specific monomials 
$X^{t_{\ell,i}-1}Y^{n_\ell-(t_{\ell,i}-1)}$, $i=1, \ldots, \tau_\ell$, of degree $n_\ell$, a choice giving rise to the notation $M_{n_\ell,T_\ell}$ for the blocks of the Bautin matrix $M_{T_{0}, \ldots, T_{\ell_d}}$ (see Notation \ref{notn.Matrice avec choix de colonnes}).

\begin{defn}\label{def.lacunarity diagram} With this notation we say that the data 
$$
(d,\ell_d, n_0, \ldots, n_{\ell_d},\tau_0, \ldots, \tau_{\ell_d},t_{0,1},\ldots t_{\ell_d,\tau_{\ell_d}}).
$$ 
are the \emph{lacunarity diagram of $Q_\lambda$}. 
\end{defn}

\begin{notn}\label{not.parties homogenes}
For any $\ell\in \llbracket 0,\ell_d\rrbracket$, we denote by $Q_{\lambda,n_\ell}$ the homogeneous part of degree $n_\ell$ of the polynomial $Q_\lambda$. With the notation above we have
$$
Q_{\lambda, n_\ell}(X,Y)=\sum_{i=1}^{\tau_\ell}
\lambda_{\ell,i}
X^{t_{\ell,i}-1}Y^{n_\ell-(t_{\ell,i}-1)}, 
$$
and $\ell_d+1$ is the number of homogeneous polynomials in $Q_\lambda$.
The number $m$ of parameters of the family $Q_\lambda$ is 
$$m
=\tau_0+\cdots+\tau_{\ell_d} 
\le (n_0+1)+\cdots + (n_{\ell_d}+1).$$ 
We set for $\ell\in \llbracket 0,\ell_d\rrbracket$, 
\begin{align*}
    \bar{\nu}_\ell =& (t_{\ell,1}-1)\nu_1+(n_\ell-(t_{\ell,1}-1))\nu_2, \\
  \mathrm{ and } \ \widetilde{\nu}_\ell =& (t_{\ell,\tau_\ell}-1)\nu_1+(n_\ell-(t_{\ell,\tau_\ell}-1))\nu_2.
\end{align*}
\end{notn}

The first lacunarity condition we consider is a condition insuring that the Bautin blocks $M_{n_\ell,T}$ do not encounter each other at some line in the Bautin matrix $M_{T_0,\ldots,T_{\ell_d}}$. Since, for $\ell\in \llbracket 0,\ell_d\rrbracket$, the block $M_{n_\ell,T_\ell}$ might contribute in $f_\lambda$ at least to degree $\nu n_\ell$ and at most to degree $n_\ell D$,  a first convenient  condition is the following arithmetic progression for our degrees $n_0, \ldots, n_{\ell_d}$
\begin{equation}\tag{$\mathcal{L}_1$}\label{eq.Lacunarity Condition 1}
\forall \ell\in \llbracket 0,\ell_d-1\rrbracket, \ \ \nu n_{\ell+1}>n_\ell D. 
\end{equation}

More accurately, with Notation \ref{not.parties homogenes}, for $\ell\in \llbracket 0,\ell_d\rrbracket$, the block $M_{n_\ell,T_\ell}$ contributes in $f_\lambda$ at least to degree $\bar{\nu}_\ell$.

Moreover we allow that two blocks \(M_{n_{\ell_1},T}\) and \(M_{n_{\ell_2},T}\) (with \(\ell_1 < \ell_2\)) encounter at a same row \(r\) if the rows above \(r\) in \(M_{n_{\ell_1},T}\) are already of full rank. Consequently the sharper, but more involved,  conditions insuring that our blocks do no intersect, are
\begin{align}\tag{$\mathcal{L}_{2a}$}\label{eq.Lacunarity Condition 1bis}
\text{in the case }\nu_1 = \nu_2,\quad &
\forall \ell\in \llbracket 0,\ell_d-1\rrbracket, \ \ 
\nu n_{\ell+1} 
>
\nu n_\ell +(k+1) (\tau_{\ell}-1),
\\
\tag{$\mathcal{L}_{2b}$}\label{eq.Lacunarity Condition 1bis with mult}
\text{in the case }\nu_1\neq\nu_2,\quad &
\forall \ell\in \llbracket 0,\ell_d-1\rrbracket, \ \ 
\bar{\nu}_{\ell+1} 
>
\widetilde{\nu}_{\ell}
\end{align}
where we recall that \(k\) is defined as the largest integer such that \((\alpha_0,\ldots,\alpha_k) = \mu \cdot (a_0,\ldots,a_k)\) for some \(\mu \in \CC\setminus\{0\}\).

We can even refine hypotheses~\eqref{eq.Lacunarity Condition 1bis} and~\eqref{eq.Lacunarity Condition 1bis with mult}, in the sense that we can allow more polynomials, under the following refined lacunarity conditions \eqref{eq.Lacunarity Condition 1ter} and \eqref{eq.Lacunarity Condition 1ter with mult}. 

Indeed in the case \(\nu_1=\nu_2\), after triangulation (obtained by a columns operation) of a Bautin block $M_{n_\ell,T_{\ell}}\cdot \mathcal{K}_{T_{\ell}}$, the leading entries of the column echelon form appear on rows \(\{ \bar{\nu}_{\ell} + i\cdot (k+1) \mid 0\leq i < \tau_{\ell}  \}\). Consequently, the following hypothesis is sufficient to ensure the different selected rows of each block do not overlap
\begin{equation}\tag{$\mathcal{L}_{3\text{a}}$}\label{eq.Lacunarity Condition 1ter}
\forall (\ell,\ell^\prime) \in \llbracket 0,\ell_d\rrbracket^2, \ \ 
\ell > \ell^\prime \text{ and } \bar{\nu}_{\ell} = \bar{\nu}_{\ell^\prime} [\text{mod } k+1] \implies 
\bar{\nu}_{\ell} > \bar{\nu}_{\ell^\prime} + (k+1)(\tau_{\ell^\prime}-1). 
\end{equation}

In the other case, \(\nu_1\neq\nu_2\), the rows which are linearly independent of the previous ones (see Remark \ref{rem:multiplicites differentes}) of a block $M_{n_\ell,T_{\ell}}$ are the rows 
\(
    \{ (t_{\ell,i}-1)\nu_1 + (n_\ell-(t_{\ell,i}-1))\nu_2 \mid i \in \llbracket 1,\tau_{\ell} \rrbracket \}
\). Given \(Q_{\lambda}\) and a pair of polynomials \((P_1,P_2)\), one can directly check whether two such rows of different blocks coincide, and the proof still works if it is not the case. 
Noticing that the gap between two such rows of a same block is always a multiple of \(\lvert \nu_1-\nu_2\rvert\), one can still set this sufficient hypothesis
\begin{equation}\tag{$\mathcal{L}_{3\text{b}}$}\label{eq.Lacunarity Condition 1ter with mult}
\forall (\ell,\ell^\prime) \in \llbracket 0,\ell_d\rrbracket^2, \ \ 
(\ell > \ell^\prime \text{ and } \bar{\nu}_{\ell} = \bar{\nu}_{\ell^\prime} [\text{mod } \lvert \nu_1-\nu_2\rvert]) \implies 
(\bar{\nu}_{\ell} > \widetilde{\nu}_{\ell^\prime}). 
\end{equation}

\begin{rem}\label{rem.zeros et parties homogenes}
 Under condition \eqref{eq.Lacunarity Condition 1}, \eqref{eq.Lacunarity Condition 1bis}, \eqref{eq.Lacunarity Condition 1bis with mult}, \eqref{eq.Lacunarity Condition 1ter}, or \eqref{eq.Lacunarity Condition 1ter with mult}, we have $\lambda\in L_b$, that is to say $Q_\lambda(P_1,P_2)=0$, if and only if for any $\ell\in \llbracket 0,\ell_d\rrbracket$,  $Q_{\lambda,n_\ell}(P_1,P_2)=~0$. 
\end{rem}
\section{Zeros of families of lacunary curves}\label{Section 5}

Using the notation of Section \ref{section 2},
let us consider a basis $(v_{i_1}, \ldots, v_{i_\sigma})$ with $S = \{i_1, \ldots, i_\sigma\} \subseteq \{0,\ldots, b\}$, of the space of linear forms vanishing on $L_b$, chosen among the elements of the family
$(v_0,\ldots, v_b)$. 

For any form $v_j(\lambda)=\sum_{i=1}^mc_j^i\lambda_i$, 
since $v_j\in \mathrm{Span}((v_k)_{k\in S})$, 
there exist $\gamma_j^{i_1}, \ldots, \gamma_j^{i_\sigma}\in \CC$ and \(\widetilde{c}\) such that  
\begin{equation}\label{eq:curve4}
v_j(\lambda)=\sum_{k \in S} \gamma_j^{k} v_{k}(\lambda), \ \ \vert\gamma_j^k\vert\leq \widetilde{c}\, \Vert v_j \Vert.
\end{equation}
where $\Vert v_j \Vert =\max_{i=1, \ldots, m}\vert c_j^i \vert$.

\begin{rem}\label{rem.sigma et dim Lb}
We have $\sigma=\dim(\mathrm{Span}((v_k)_{k\in S})=m-\dim(L_b)$. 
\end{rem}

\begin{notn}\label{no:c}
We denote by $c=c(f_\lambda,S)>0$ the minimum of the constants $\widetilde{c}$ satisfying (\ref{eq:curve4}).
\end{notn}

With the notation of Section \ref{section 1}, in  this section we estimate from above the number of zeros of $f_\lambda$ in a disc, following \cite{Roy.Yom} and \cite{CoYo1}. To fix notation, and for the convenience of the reader, we recall and adapt the method to our situation. 
\begin{defn}[Bernstein class]\label{def.Berstein1}
Let $r>0$, $\alpha\in ]0,1[$, $K>0$ and $f$ be a nonzero analytic function on a neighbourhood of the closed disc $\bar{D}_r$. We say that $f$ belongs to the \emph{Bernstein class $B^1_{r,\alpha,K}$} when 
$$
\dfrac{\max_{\bar{D}_r}{\vert f \vert }}{\max_{\bar{D}_{\alpha r}}{\vert f \vert }}\le K.
$$
\end{defn}

Now let the family $Q_i, i=1, \ldots, m$ be given, and let the Bautin index $b$ as in Section \ref{section 1}. 
For any \(r < R\) and any \(\alpha \in \left]0,1\right[\), we denote by \(\mu\) the maximum of \(\lvert f_\lambda\rvert\) over the closed disk \(\bar{D}_{\alpha r}\). By~\eqref{eq:curve2.1} we get for all \(i\), \(\lvert v_i\rvert \leq \mu / (\alpha r)^i\). Hence, following the proof of Theorem 2.1.3 in~\cite{Roy.Yom}, by~\eqref{eq:curve4} and using again~\eqref{eq:curve2.1},
\begin{align}
    \label{eq.maj gamma}
    \max_{\bar{D}_{r}}(\lvert f_\lambda \rvert)
    & \le \sum_{i=0}^b \lvert v_i(\lambda)\rvert r^i + \sum_{j \ge b+1} \lvert v_j(\lambda)\rvert r^j \\
    \nonumber
    & \le \mu\sum_{i=0}^b \frac{r^i}{(\alpha r)^i} + \sum_{j \ge b+1}  \sum_{k \in S} \lvert\gamma_j^k\rvert \lvert v_k(\lambda)\rvert R^j \left(\frac{r}{R}\right)^j \\
    \nonumber
    & \le \frac{\mu}{\alpha^b}\frac{1-\alpha^{b+1}}{1-\alpha} + \sum_{j \ge b+1} \sum_{k \in S} c \left(R^j \max_{i=1,\ldots,m} \lvert c_j^i\rvert\right) \lvert v_k(\lambda)\rvert \left(\frac{r}{R}\right)^j \\
    \nonumber
    & \le \frac{\mu}{\alpha^b}\frac{1-\alpha^{b+1}}{1-\alpha} + B_R c \sum_{k \in S} \lvert v_k(\lambda)\rvert \sum_{j \ge b+1}  \left(\frac{r}{R}\right)^j \\
    \nonumber
    & \le \frac{\mu}{\alpha^b}\frac{1-\alpha^{b+1}}{1-\alpha} + B_R c \frac{(r/R)^{b+1}}{1-r/R} \mu \sum_{k \in S} \frac{1}{(\alpha r)^k} \\
    \nonumber
    & \le \frac{\mu}{\alpha^b} \left( 1 + \frac{\alpha(1-\alpha^b)}{1-\alpha} + \frac{B_R c r}{R^{b+1}(1-r/R)} \sum_{k\in S}(\alpha r)^{b-k} \right).
\end{align}

Let us introduce \(\bar{\sigma}_{\rho} = \sum_{k\in S}\rho^{b-k}\). We have 
\begin{equation}
    \label{eq:ub_sigma}
    \bar{\sigma}_{\rho} \le
    \begin{cases}
    \frac{1-\rho^\sigma}{1-\rho} & \text{if }\rho < 1,
    \\
    \sigma & \text{if }\rho = 1,
    \\
    \rho^{b+1}\frac{1-1/\rho^\sigma}{\rho - 1} & \text{if }\rho >1.
    \end{cases}
\end{equation}

It follows that for $\lambda\not\in L_b$, $f_\lambda\in B^1_{r,\alpha,K}$ with 
\begin{equation}\label{eq.K}
K=\frac{1}{\alpha^b}\left(1+\frac{\alpha(1-\alpha^b)}{1-\alpha}+
\frac{B_R c r \bar{\sigma}_{\alpha r}}{R^{b+1}(1-r/R)} \right).
\end{equation}

We can now give in Theorem \ref{thm:zroes.baut} (in the same way as \cite[Lemma 2.2.3]{Roy.Yom}) a bound from above on the number of zeros of $f_\lambda$ in $\bar{D}_{R/4}$, using the following Jensen–Nevanlinna-type inequality 
(see \cite[Lemma 1]{VdP}, or \cite{Yako} for  other proofs of 
 statements comparable to \eqref{num.radius} of Theorem \ref{thm:zroes.baut} below).
\begin{thm}\label{t.VanderPoorten}
Let $f_\lambda\in B^1_{r,\alpha, K}$, then the number of zeros of $f_\lambda$ in $\bar{D}_{\alpha r}$ is at most $\dfrac{\log K}{\log((1+\alpha^2)/2\alpha)}$.
\end{thm}

\begin{thm}\label{thm:zroes.baut}
Zeros being counted with multiplicity, we have the following uniform bounds.
\begin{enumerate}
\item The maximal number of zeros of  $f_{\lambda}$, with respect to $\lambda\not \in L_b$, in the disk $\bar{D}_{\frac{R}{4}}$, is at most
$$ 5b\log2+5\log(2R^b + B_R c \bar{\sigma}_{R/4}) - 5b\log(R).$$

\item\label{num.radius} The maximal number of zeros of $f_\lambda$, with respect to $\lambda\not \in L_b$, is at most $b$
in $\bar{D}_\rho$, where
$$
\rho \le \frac{R}{8^b\max(2,B_R c \bar{\sigma}_{R/4}/R^b) }.
$$

\end{enumerate}
\end{thm}

\begin{proof} Since for $\lambda\not\in L_b$, $f_\lambda\in B^1_{r,\alpha, K}$
for any $r<R$ and any $\alpha\in ]0,1[$, and for $K$ as in \eqref{eq.K}, choosing $r=R/2$ and $\alpha=1/2$, one gets $K\le 2^b(2+B_R c \bar{\sigma}_{R/4}/R^b\})$. By Theorem \ref{t.VanderPoorten}, one obtains that the number of zeros in $\bar{D}_{\frac{R}{4}}$ is at most 
$$
\frac{1}{\log5/4}\left( b\log2+\log\left(2+\frac{B_R c \bar{\sigma}_{R/4}}{R^b}\right)
\right)$$
$$
\le 5b\log2+5\log(2R^b+B_R c \bar{\sigma}_{R/4}) - 5b\log(R).
$$
The second bound in the statement comes from 
\cite[Lemma 2.2.3]{Roy.Yom}. 
\end{proof}

The next step consists now in the search of a bound from above for the constants $b,c$ and $\sigma$ appearing in Theorem \ref{thm:zroes.baut}. All the information we need is encoded in the rank $\sigma$ matrix
$ M(f_\lambda)=(c^i_k), \ k=0,\ldots,b, \ i=1,\dots,m$, with $b+1$ lines and $m$ columns.

\begin{notn}\label{not: small delta}
We will denote by $\delta>0$ the maximum of the absolute value of all  nonzero minor
determinants of size $\sigma$ of the Bautin matrix $M(f_\lambda)$.
\end{notn}
We  give below an estimate of the constant $c$, and then of $\mathcal{Z}(f_\lambda)$.

\begin{prop}[{\cite[Proposition 2.11]{CoYo1}}]\label{Prop:Est.c}
Let $f_\lambda$ be given as above, with $\lambda\not\in L_b$. Then
$$
c(f_\lambda) \le
 \sigma\frac{ (B_R \sqrt \sigma)^{\sigma-1} }{\delta R^{\beta(\sigma-1)} },$$
 where  $\beta=b$ if $R\le 1$ and $\beta=\sigma/2$ if  $R\ge 1.$ 

In the disk $\bar{D}_{\frac{R}{4}}$, the maximal number $\mathcal{Z}(f_\lambda)$ of zeros of the family $f_\lambda$, with respect to the parameter $\lambda$,  satisfies
\begin{align*}
\mathcal{Z}(f_\lambda) & \le 
5b\log2 
+ 5\log\left( 2+  
\frac{ \sigma^{\frac{\sigma+1}{2}}B_R^\sigma\bar{\sigma}_{R/4} } { \delta R^{ b \sigma }}
\right)
\ \hbox{ if } R\le 1, \\
\hbox{ and } \ \mathcal{Z}(f_\lambda) & \le 
5b\log2
+ 5\log\left( 2+
\frac{ \sigma^{\frac{\sigma+1}{2}}B_R^\sigma\bar{\sigma}_{R/4}}{ \delta R^{b+ \frac{\sigma}{2}(\sigma-1) } }
\right)
\ \hbox{ if } R\ge 1.
\end{align*}

\end{prop}
\begin{proof}
The proof of the bound on $c(f_\lambda)$ is the same as the one of \cite[Proposition 2.11]{CoYo1}. The proof of the bounds on $\mathcal{Z}(f_\lambda)$ is then a consequence of Theorem \ref{thm:zroes.baut}.
\end{proof}

\begin{rem}\label{rem.R=B=1}
Since $f_\lambda(z)=0$ for any $z\in \bar{D}_R$ is equivalent to
$f_\lambda(Rz)=0$ for any $z\in \bar{D}_1$, a uniform bound with respect to $\lambda\not\in L_b(f_\lambda(Rz))$ on the zeros of $f_\lambda(Rz)$ in $\bar{D}_{\frac{1}{4}}$ gives a uniform bound with respect to $\lambda\not\in L_b(f_\lambda(z))$ on the zeros of $f_\lambda(z)$ in $\bar{D}_{\frac{R}{4}}$. 

If we denote by $\widetilde{P}$ the polynomial mapping $\dfrac{P}{M_R}$, with notation \eqref{eq.Polynomial P} and \eqref{eq.Max P}, 
we have a bijection between $L_b(\widetilde{f}_\lambda=Q_\lambda(\widetilde{P}))$ and $L_b(f_\lambda)$ sending 
$(\lambda_i)_{i=1, \ldots, m}$ to $\left(\dfrac{\lambda_i}{M_R^{d_i}}\right)_{i=1, \ldots, m}$, where 
$d_i$ is the degree of the monomial $Q_i$ corresponding to the parameter $\lambda_i$. In consequence a uniform bound, with respect to $\lambda~\not\in~L_b(\widetilde{f}_\lambda)$, on the zeros of $\widetilde{f}_\lambda=Q_\lambda(P)$ in $\bar{D}_{\frac{R}{4}}$ gives a uniform bound, with respect to $\lambda\not\in L_b(f_\lambda)$, on the zeros of $f_\lambda$ in $\bar{D}_{\frac{R}{4}}$. Notice that in this situation 
since $\widetilde{P}$ has its components bounded by $1$ on $\bar{D}_R$, for any 
monomial $Q_i$, $Q_i(P)$ is also bounded by $1$ on $\bar{D}_R$.

In conclusion to bound the number of zeros of $f_\lambda$, with respect to $\lambda\not\in L_b(f_\lambda)$, we can always consider that $R=B_R=1$, up to changing 
the polynomial mapping $P(X)$ into the polynomial mapping (of same degree $D$) $P(RX) / \max_{z\in \bar{D}_1, i=1, 2} \vert P_i(Rz)\vert $.
\end{rem}

Remark \ref{rem.R=B=1} allows us to only investigate the case $R=B_R=1$ without loss of generality. By Proposition \ref{Prop:Est.c} we obtain in this case (since by~\eqref{eq:ub_sigma}, we have \(\bar{\sigma}_{1/4} \le 4/3\)) the following simple bound for $\mathcal{Z}(f_\lambda)$. 

\begin{cor}\label{cor.nombre de zeros avec B=R=1}
Let a family $f_\lambda$ be given as in \eqref{eq:f.lambda}, with $R=B_R=1$, then the maximal number $\mathcal{Z}({f_\lambda})$ of zeros of $f_\lambda$ in $\bar{D}_{\frac{1}{4}}$, with respect to $\lambda\not\in L_b$, satisfies
$$
\mathcal{Z}(f_\lambda)\le 5b\log2 + 5\log\left( 2+\frac{4\sigma^{\frac{\sigma+1}{2}}}{3\delta}  \right).
$$
Furthermore,  the maximal number of zeros of $f_\lambda$, with respect to $\lambda\not \in L_b$, is at most $b$
in $\bar{D}_\rho$, where
$$
\rho = \frac{1}{8^b\max\{2,4c/3\}  } \ge 
\frac{1}{8^b}\min\left\{ \frac{1}{2}, \frac{3\delta}{4\sigma^{\frac{\sigma+1}{2}}}   \right\}.
$$
\end{cor}

\begin{rem}\label{rem.c condition 1}
    Under condition~\eqref{eq.Lacunarity Condition 1}, we can improve on the bound on \(\mathcal{Z}(f_{\lambda})\) of Corollary~\ref{cor.nombre de zeros avec B=R=1} by giving a better estimate on the constants \(\gamma_{j}^k\), where \(j > b\), in~\eqref{eq.maj gamma}. 
    Indeed, we can improve on arguments of the proof of Proposition~\ref{Prop:Est.c} in the following way. 
    By condition~\eqref{eq.Lacunarity Condition 1}, for \(j > b\), the linear form \(v_j\) is a linear combination of the \(v_k\) where \(k\) is in the set \(S_{\ell_d}\) which is the intersection of \(S\) (introduced in the beginning of this section) and of the indices \(k\) satisfying \(k \ge \bar{\nu}_{\ell_d}-1\) (that is to say, the coefficients of \(v_k\) are given by the last block \(M_{n_{\ell_d},T_{\ell_d}}\) of the Bautin matrix \(M_{T_0,\ldots,T_{\ell_d}}\), see Notation~\ref{notn.Matrice avec choix de colonnes}). Working with these parameters in the proof of Proposition~\ref{Prop:Est.c}, gives the same bound with \(\sigma\) replaced by \(\tau_{\ell_d}\) and \(\delta\) replaced by the determinant of the upper \(\tau_{\ell_d} \times \tau_{\ell_d}\) minor of \(M_{n_{\ell_d},T_{\ell_d}}\).

\end{rem}

In view of Corollary \ref{cor.nombre de zeros avec B=R=1}, we compute in this section the Bautin index $b$, the dimension $\sigma$ and the determinant $\delta$. All those data are related to the Bautin matrix $M_{ T_{0}, \ldots, T_{\ell_d}}$ (see Notation \ref{notn.Matrice avec choix de colonnes}), and thanks to \eqref{eq.Lacunarity Condition 1}, \eqref{eq.Lacunarity Condition 1bis}, \eqref{eq.Lacunarity Condition 1bis with mult}, \eqref{eq.Lacunarity Condition 1ter} or \eqref{eq.Lacunarity Condition 1ter with mult}, they can be computed on the blocks $M_{n_\ell}$, $\ell=0, \ldots, \ell_d$. Moreover for each block  $M_{n_\ell}$ we can perform our computation on the reduced matrix $M_{n_\ell,T_{\ell}}\cdot \mathcal{K}_{T_{\ell}}$, since our triangulation process does not change the rank nor~$\delta$. 

\begin{prop}[Value of $\sigma$]\label{prop.valeur de sigma}
Under condition \eqref{eq.Lacunarity Condition 1}, \eqref{eq.Lacunarity Condition 1bis},  \eqref{eq.Lacunarity Condition 1bis with mult}, \eqref{eq.Lacunarity Condition 1ter} or \eqref{eq.Lacunarity Condition 1ter with mult}, we have  $$\dim(L_b)=0 \ \mathrm{ and } \ \sigma=m=\tau_0+\cdots+\tau_{\ell_d},$$ unless $P_1$ and $P_2$ are proportional, in which case 
$$\dim(L_b)=\tau_0+\cdots+\tau_{\ell_d}-(\ell_d+1) \ \mathrm{and} \ \sigma=\ell_d+1.$$ 

In particular, in case $P_1$ and $P_2$ are not proportional, condition
\eqref{eq.Lacunarity Condition 1}, \eqref{eq.Lacunarity Condition 1bis}, \eqref{eq.Lacunarity Condition 1bis with mult}, \eqref{eq.Lacunarity Condition 1ter} or \eqref{eq.Lacunarity Condition 1ter with mult} guarantees that $(P_1,P_2)(\CC)\cap \{Q_\lambda=0\}$ is a finite set. 
\end{prop}

\begin{proof}
In case $P_1$ and $P_2$ are proportional, by Remark \ref{rem.proportionalite},
any block  $M_{n_\ell,T_{\ell}}$ contributes for only one linear form in a basis of linear forms of $\mathrm{Span} (v_0,\ldots, v_b)$. 
It follows that $\sigma$ is the number of blocks, that is $\ell_d+1$.

Finally, in case $P_1$ and $P_2$ are not proportional, no parameter 
$\lambda\not=0$ can cancel $Q_\lambda(P_1,P_2)$, showing that $\dim(L_b)=0$, and by Remark \ref{rem.sigma et dim Lb}, that $\sigma=m-\dim(L_b)=\tau_0+\cdots +\tau_{\ell_d}$. 
\end{proof}

\begin{prop}[Value of $b$]\label{prop.valeur de b}
Under condition \eqref{eq.Lacunarity Condition 1}, \eqref{eq.Lacunarity Condition 1bis}, \eqref{eq.Lacunarity Condition 1bis with mult}, \eqref{eq.Lacunarity Condition 1ter} or \eqref{eq.Lacunarity Condition 1ter with mult}, with the notation of Notation \ref{not.parties homogenes} and \eqref{eq.singular condition}, we have 
\begin{equation}\label{eq.valeur de b multiplicites differentes}
\mathrm{for}  \ 
\nu_1\not=\nu_2, \ \ 
b= \max_{\ell \in \llbracket 0,\ell_d\rrbracket}\left(\widetilde{\nu}_{\ell}\right),
\end{equation}
\begin{equation}\label{eq.valeur de b multiplicites egales}
\mathrm{and \ for} \ \nu_1=\nu_2=\nu,  \ \ b=\max_{\ell \in \llbracket 0,\ell_d\rrbracket}\left(\nu n_{\ell}+(k+1)(\tau_{\ell}-1)\right), 
\end{equation}
unless $P_1$ and $P_2$ are proportional, in which case 
\begin{equation}\label{eq.valeur de b P1 et P2 proportionnels}
b= \bar{\nu}_{\ell_d}=\nu d.     
\end{equation}

\end{prop}

Notice that under conditions \eqref{eq.Lacunarity Condition 1}, \eqref{eq.Lacunarity Condition 1bis}, or \eqref{eq.Lacunarity Condition 1bis with mult}, both maxima are achieved for \(\ell = \ell_d\).

\begin{proof}
The Bautin index $b$ is the lowest row $b+1$ in the Bautin matrix $M_{T_{0}, \ldots, T_{\ell_d}}$, corresponding to the coefficients of the linear form $v_b$, where $v_b\not\in \mathrm{Span}(v_0, \ldots, v_{b-1})$. 
Under condition \eqref{eq.Lacunarity Condition 1}, \eqref{eq.Lacunarity Condition 1bis}, or \eqref{eq.Lacunarity Condition 1bis with mult} this last row has to be found in the last block $M_{n_{\ell_d},T_{\ell_d} }$
of $M_{T_{0}, \ldots, T_{\ell_d}}$. 
This last block starts at row $\bar{\nu}_{\ell_d}+1$ of  $M_{T_{0}, \ldots, T_{\ell_d}}$, corresponding to the coefficients of $v_{\bar{\nu}_{\ell_d}}$. 
\begin{itemize}
    \item[-] 
Let us first assume that $P_1$ and $P_2$ are not proportional. In this situation, we have two sub-cases. 
\begin{itemize}
\item[-] The first one is the case where the multiplicities $\nu_1$ and $\nu_2$ of $P_1$ and $P_2$ are not the same. Then by Remark \ref{rem:multiplicites differentes} the last row in $M_{n_{\ell_d},T_{\ell_d} }$ independent of the highest ones is the row corresponding to $v_{\widetilde{\nu}_{\ell_d}}$. So $b=\widetilde{\nu}_{\ell_d}$. 
\item[-]
 The second sub-case is the case where $\nu_1=\nu_2$. In this situation, let us denote 
by $k\in \llbracket 0, \min(D_1,D_2)-\nu \rrbracket$ the largest integer such that
the k-jets of $P_1$ and $P_2$ are proportional, accordingly to \eqref{eq.singular condition}. By Proposition \ref{prop.sous-diagonales nulles}, the row $(k+1)\tau_{\ell_d}-k$   of the block  $M_{n_{\ell_d},T_{\ell_d} }$, corresponding to the coefficients of the linear form 
$v_{\bar{\nu}_{\ell_d}+(k+1)\tau_{\ell_d}-k-1}$, is independent of the 
other highest rows, and the other lowest rows are not independent of those $(k+1)\tau_{\ell_d}-k$ first ones. In conclusion, 
$b=\bar{\nu}_{\ell_d}+(k+1)\tau_{\ell_d}-k-1 = \nu n_{\ell_d}+(k+1)(\tau_{\ell_d}-1)$.
\end{itemize}
\item[-]  In case  $P_1$ and $P_2$ are proportional, as noticed in Remark~\ref{rem.proportionalite},
all columns of  $M_{n_{\ell_d},T_{\ell_d} }$ but the first one have only zero for entries, meaning that in this situation $ b= \bar{\nu}_{\ell_d}=n_{\ell_d}\nu= \nu d$.  
\end{itemize}
Under condition \eqref{eq.Lacunarity Condition 1ter} or \eqref{eq.Lacunarity Condition 1ter with mult}, the proof is similar by replacing \(\ell_d\) by respectively \(\text{argmax}_{\ell}(\widetilde{\nu}_\ell)\) or \(\text{argmax}_{\ell}(\bar{\nu}_\ell+(k+1)\tau_{\ell})\).
\end{proof}

We finally give a bound from above of $\delta$, the maximum of the absolute value of all nonzero minor of size $\sigma$ of $M_{T_{0}, \ldots, T_{\ell_d}}$ (see Notation \ref{not: small delta}). For this we first fix some notation. 

\begin{notn}\label{not.barre} In the case where $\nu_1=\nu_2$, 
with notation \eqref{eq.singular condition}, we set
\begin{align*}\label{eq.delta k barre}
\bar{\delta}_k  = \alpha_{k+1}-\mu a_{k+1}, \ \
\ell_e  = \text{argmax}_\ell(\bar{\nu}_\ell+(k+1)\tau_\ell).
\end{align*}
In the case where \(\nu_1\neq\nu_2\), we set \(\ell_e = \text{argmax}_{\ell}(\widetilde{\nu}_{\ell})\). And in both cases, we set, for $\ell\in \llbracket 0,\ell_d\rrbracket$, 

\begin{align*}
\ \ 
\bar{t}_\ell &=\sum_{i\in \llbracket 1, \tau_\ell \rrbracket} t_{\ell,i}, 
\ \ 
\bar{t}=\sum_{\ell\in \llbracket 0, \ell_d \rrbracket} \bar{t}_\ell,
\ \ 
\bar{t}'_1=\sum_{\ell\in \llbracket 0, \ell_d \rrbracket} \bar{t}_{\ell,1}, \\
\bar{n}& =\sum_{i\in \llbracket 0, \ell_d \rrbracket} n_\ell, 
\ \ 
\overline{\tau n}=\sum_{\ell\in \llbracket 0, \ell_d \rrbracket} \tau_\ell n_\ell, \\
C_\ell & =
\frac{
1  }
{\prod_{i\in \llbracket 1, \tau_\ell \rrbracket} (i-1)! 
\prod_{i\in \llbracket 1, \tau_\ell \rrbracket} \vert K_{\{t_{\ell,1},\ldots,t_{\ell, i}\},t_{\ell,i}}\vert }, \ \ 
\bar{C}=\prod_{\ell\in \llbracket 0,\ell_d \rrbracket}C_\ell.
\end{align*}
We finally recall that, $m$ being the number of parameters of our family $Q_\lambda$, 
\begin{equation*}\label{eq.tau barre}
m=\sum_{\ell=0}^{\ell_d}\tau_{\ell}, \  
\mathrm{ and \ we \ set} \  \
\bar{\tau}=\sum_{\ell\in \llbracket 0,\ell_d\rrbracket} \frac{\tau_\ell(\tau_\ell+1)}{2}, 
\ \ 
\bar{\tau}'=\sum_{\ell\in \llbracket 0,\ell_d\rrbracket} \frac{\tau_\ell(\tau_\ell-1)}{2}
. 
\end{equation*}
\end{notn}

\begin{prop}[Value of $\delta$]\label{prop.valeur de delta}
Under condition \eqref{eq.Lacunarity Condition 1}, \eqref{eq.Lacunarity Condition 1bis}, \eqref{eq.Lacunarity Condition 1bis with mult}, \eqref{eq.Lacunarity Condition 1ter} or \eqref{eq.Lacunarity Condition 1ter with mult}, with the notation of Notation \ref{not.parties homogenes}, 
\ref{not.barre} and \eqref{eq.singular condition},
we have for  $\nu_1\not=\nu_2$,  
\begin{equation}\label{eq.valeur de delta multiplicites differentes}
\delta\ge \vert \alpha_0\vert^{\bar{t}-m
}
\vert a_0\vert^{\overline{\tau n}-(\bar{t}-m)
} ,
\end{equation}
and for $\nu_1=\nu_2$,   
\begin{equation}\label{eq.valeur de delta multiplicites egales}
\delta\ge
\bar{C} 
\vert \mu\vert^{\bar{t}-\bar{\tau}}
\vert a_0\vert^{  \overline{\tau n}-\bar{\tau}'  } 
\vert \bar{\delta}_k\vert^{
\bar{\tau}' } , 
\end{equation}
unless $P_1$ and $P_2$ are proportional, in which case 
\begin{equation}\label{eq.valeur de delta P1 et P2 proportionnels}
\delta
\ge 
\vert\mu\vert^{\bar{t}'_1 - (\ell_d+1)} 
\vert a_0\vert^{\bar{n}}.     
\end{equation}

Under condition~\ref{eq.Lacunarity Condition 1}, more accurately, instead of \(\delta\) one can take in Corollary~\ref{cor.nombre de zeros avec B=R=1} the following bounds: for  $\nu_1\not=\nu_2$, 
\begin{equation}\label{eq.valeur de delta multiplicites differentes cas L1}
\vert \alpha_0\vert^{\bar{t}_{\ell_d}-\tau_{\ell_d}}
\vert a_0\vert^{\tau_{\ell_d} n_{\ell_d} +\tau_{\ell_d} - \bar{t}_{\ell_d}},
\end{equation}
 and for  $\nu_1=\nu_2$,
\begin{equation} \label{eq.valeur de delta multiplicites egales cas L1}
C_{\ell_d} 
\vert \mu\vert^{\bar{t}_{\ell_d}-\frac{\tau_{\ell_d}(\tau_{\ell_d}+1)}{2} }
\vert a_0\vert^{\tau_{\ell_d} n_{\ell_d}-\frac{\tau_{\ell_d}(\tau_{\ell_d}-1)}{2} } 
\vert \bar{\delta}_k\vert^{\frac{\tau_{\ell_d}(\tau_{\ell_d}-1)}{2} },
\end{equation}
unless $P_1$ and $P_2$ are proportional, in which case 
\begin{equation} \label{eq.valeur de delta P1 et P2 proportionnels cas L1}
\lvert\mu\rvert^{t_{\ell_d,1}-1}\lvert a_0\rvert^{n_{\ell_d}}.
\end{equation}
\end{prop}

\begin{proof} Under condition \eqref{eq.Lacunarity Condition 1},  \eqref{eq.Lacunarity Condition 1bis},  \eqref{eq.Lacunarity Condition 1bis with mult}, \eqref{eq.Lacunarity Condition 1ter} or \eqref{eq.Lacunarity Condition 1ter with mult}, a nonzero minor of maximal size $\sigma$ of the Bautin matrix $M_{T_{0}, \ldots, T_{\ell_d}}$ is given as the product, over $\ell\in \llbracket 0,\ell_d\rrbracket$, 
of nonzero minors $\delta_\ell$ of maximal size chosen in each block $M_{n_\ell,T_{\ell}}$ (as long as the rows of the minor \(\delta_\ell\) do not overlap the blocks \(M_{n_{\ell'},T_{\ell'}}\) for \(\ell'>\ell\)). 
Moreover if needed, we can choose $\delta_\ell$ in the reduced matrix $M_{n_\ell,T_{\ell}}\cdot \mathcal{K}_{T_{\ell}}$. 
\begin{itemize}
    \item[-] 
Let us first assume that $P_1$ and $P_2$ are not proportional. In this situation, as in the proof of Proposition \ref{prop.valeur de b}, we have two sub-cases. 
\begin{itemize}
\item[-] The first one is the case where the multiplicities $\nu_1$ and $\nu_2$ of $P_1$ and $P_2$ are not the same. Let us fix  $\ell\in \llbracket 0,\ell_d\rrbracket$. 
Then by Remark \ref{rem:multiplicites differentes}, the columns of the block  $M_{n_\ell,T_{\ell}}$ are shifted down by zero-columns, and the first upper nonzero coefficient in column $i\in \llbracket 1, \tau_{\ell} \rrbracket$, is $\alpha_0^{t_{\ell,i}-1}a_0^{n_\ell- (t_{\ell,i}-1) }$. We can therefore choose
$$
\delta_\ell=
\vert \alpha_0\vert^{\sum_{i\in \llbracket 1, \tau_\ell \rrbracket}t_{\ell,i}-\tau_{\ell}}
\vert a_0\vert^{\tau_\ell n_\ell- \left(\sum_{i\in \llbracket 1, \tau_\ell \rrbracket} t_{\ell,i}-\tau_\ell \right) },
$$
giving, with Notation \ref{not.barre}, the following bound from below for $\delta$
$$
\prod_{\ell\in \llbracket 0,\ell_d\rrbracket} \delta_\ell
=
\vert \alpha_0\vert^{\bar{t}-m}
\vert a_0\vert^{\overline{\tau n}-(\bar{t}-m)}. $$ 
\item[-]
 The second sub-case is the case where $\nu_1=\nu_2$. In this situation, let us denote 
by $k\in \llbracket 0, \min(D_1,D_2)-\nu \rrbracket$ the largest integer such that
the k-jets of $P_1$ and $P_2$ are proportional, accordingly to \eqref{eq.singular condition}.
Let us fix  $\ell\in \llbracket 0,\ell_d\rrbracket$. 
By Proposition \ref{prop.sous-diagonales nulles}, in the block $M_{n_\ell,T_{\ell}}\cdot \mathcal{K}_{T_{\ell}}$ we can choose
the rows $(k+1)i-k$, for $i\in \llbracket 1, \tau_\ell\rrbracket$ to compute $\delta_\ell$. We obtain by 
\eqref{eq.coefficients jets egaux a ordre k}, with Notation \ref{not.barre},
$$
\delta_\ell=
C_\ell 
\vert \mu\vert^{\bar{t}_\ell-\frac{\tau_\ell(\tau_\ell+1)}{2} }
\vert a_0\vert^{\tau_\ell n_{\ell}-\frac{\tau_\ell(\tau_\ell-1)}{2} } 
\vert \bar{\delta}_k\vert^{\frac{\tau_\ell(\tau_\ell-1)}{2} },
$$
giving the following bound from below for $\delta$
$$
\prod_{\ell\in \llbracket 0,\ell_d\rrbracket} \delta_\ell
= \bar{C}
\vert \mu \vert^{\bar{t}-\bar{\tau}}
\vert a_0\vert^{\overline{\tau n}-\bar{\tau}' } 
\vert \bar{\delta}_k\vert^{\bar{\tau}' }
.
$$

\end{itemize}
\item[-]  In case  $P_1$ and $P_2$ are proportional, using Remark \ref{rem.proportionalite},
all columns of  $M_{n_{\ell_d},T_{\ell_d} }$ but the first one have only zero for entries. We then choose, for each  $\ell\in \llbracket 0,\ell_d\rrbracket$,
in each block $M_{n_\ell,T_{\ell}}$, the first row, having for only nonzero coefficient $\mu^{t_{\ell,1}-1}a_0^{n_\ell}$ by \eqref{eq.coefficient en haut a gauche}. This gives, with Notation \ref{not.barre}, the following bound from below for $ \delta $
$$
\prod_{\ell\in \llbracket 0,\ell_d\rrbracket}
\vert \mu\vert^{t_{\ell,1}-1}a_0^{n_\ell}
=
\vert\mu\vert^{\bar{t}'_1 - (\ell_d+1)} 
\vert a_0\vert^{\bar{n}}. 
$$
\end{itemize}

Under condition~\ref{eq.Lacunarity Condition 1}, we can take \(\delta_{\ell_d}\) in each case above as the final bound.
\end{proof}

We now combine Propositions \ref{prop.valeur de sigma}, \ref{prop.valeur de b} and \ref{prop.valeur de delta} with Corollary \ref{cor.nombre de zeros avec B=R=1} to give bounds for the number of intersection points, near the origin, between the zero set of a  family of lacunary polynomials and a given polynomially parametrized algebraic curve. Those bounds are given in terms of the data that explicitly quantify
the lacunarity: the number of monomials appearing in the family, the degree of those monomials, and the particular choices made among monomials of a given degree. 
\begin{thm}\label{thm.Bezout Bound for lacunary polynomials} Let a family $f_\lambda$ be given as in \eqref{eq:f.lambda}, with $R=B_R=1$ and \(m \ge 2\).
Under condition \eqref{eq.Lacunarity Condition 1}, \eqref{eq.Lacunarity Condition 1bis}, \eqref{eq.Lacunarity Condition 1bis with mult} \eqref{eq.Lacunarity Condition 1ter} or \eqref{eq.Lacunarity Condition 1ter with mult}, with the notation of Notation \ref{not.parties homogenes}, \ref{not.barre} and \eqref{eq.singular condition}, denoting by  $\mathcal{Z}({f_\lambda})$ the maximal number of zeros of $f_\lambda$ in $\bar{D}_{\frac{1}{4}}$, with respect to $\lambda\not\in L_b$, we have, for $P_1$ and $P_2$ not proportional,  
\begin{enumerate}
    \item\label{item.Bezout Bound 1}
   in the case where $\nu_1\not=\nu_2$, 
        \begin{align*}
   \mathcal{Z}(f_\lambda) & \le  
   5(\widetilde{\nu}_{\ell_e}+1)\log2 
   + \frac{5}{2}(m+1)\log m\\
  & +5(\bar{t}-m)\log\left( \frac{1}{\vert \alpha_0 \vert}\right)
   +5(\overline{\tau n}- (\bar{t}-m))\log\left( \frac{1}{\vert a_0 \vert}
   \right) \\
   & = O(dm),
\end{align*}
    and the maximal number of zeros of $f_\lambda$, with respect to $\lambda\not=0$, is at most 
    $\widetilde{\nu}_{\ell_e}$
in $\bar{D}_\rho$, where
\[
\rho  = 
\frac{3}{4\cdot 8^{\widetilde{\nu}_{\ell_e}}}
 \frac{\vert \alpha_0\vert^{\bar{t}-m}\vert a_0\vert^{\overline{\tau n}-(\bar{t}-m)}}{m^{\frac{m+1}{2}}}.
 \]
         \item\label{item.Bezout Bound 2} In the case where $\nu_1=\nu_2$, 
         \begin{align*}
          \mathcal{Z}(f_\lambda) & \le    5(\bar{\nu}_{\ell_e}+(k+1)(\tau_{\ell_e}-1))\log2 \\
          & + 5\log\left( 2+\frac{4m^{\frac{m+1}{2}}}{ 3\bar{C} \vert \mu\vert^{\bar{t}-\bar{\tau}}
\vert a_0\vert^{  \overline{\tau n}-\bar{\tau}'} 
\vert \bar{\delta}_k\vert^{\bar{\tau}' } }  \right) \\
    & = O(dm),
         \end{align*}
         
 and the maximal number of zeros of $f_\lambda$, with respect to $\lambda\not=0$, is at most 
    $\nu n_{\ell_e}$
in $\bar{D}_\rho$, where
$$
\rho  = \frac{3}{4\cdot 8^{\nu n_{\ell_e}+(k+1)(\tau_{\ell_e}-1) } } \min\left\{\frac{2}{3}, 
 \frac{\bar{C} 
\vert \mu\vert^{\bar{t}-\bar{\tau}}
\vert a_0\vert^{  \overline{\tau n}-\bar{\tau}'  } 
\vert \bar{\delta}_k\vert^{
\bar{\tau}' }}{m^{\frac{m+1}{2}}} 
\right\}.
$$

\end{enumerate}

\end{thm}

\begin{rem}\label{rem.symmetry}
The rational curves \((P_1,P_2)\) and \((P_2,P_1)\) are symmetric in the plane. Given a lacunary polynomial \(Q_{\lambda}(X,Y)\), the number of intersection points between \((P_1,P_2)\) and \(Q_{\lambda}=0\) is the same as the number of intersection points between \((P_2,P_1)\) and \(Q_{\lambda}(Y,X)=0\). Exchanging \(X\) and \(Y\) in \(Q_{\lambda}\) changes the numbering \((t_{0,1},\ldots,t_{\ell_d,\tau_{\ell_d}})\) of the columns in the following way: In a block \(M_{n_{\ell},T_{\ell}}\), the number \(\bar{t}_{\ell}\) becomes \((n_{\ell}+1)\tau_{\ell} - \bar{t}_{\ell}\). The bounds on \(\mathcal{Z}(f_{\lambda})\) and \(\rho\) in Theorem~\ref{thm.Bezout Bound for lacunary polynomials} are unaffected under this transformation.  
\end{rem}

\begin{proof}
We consider the two cases separately. 

\begin{enumerate}
    \item
Let us first consider the case  where $\nu_1\not=\nu_2$.

By Proposition \ref{prop.valeur de sigma} we have 
$\sigma=m$, 
       by \eqref{eq.valeur de b multiplicites differentes} we have 
$b=\widetilde{\nu}_{\ell_e}$, and by \eqref{eq.valeur de delta multiplicites differentes}
$\delta\ge \vert \alpha_0\vert^{\bar{t}-m}\vert a_0\vert^{\overline{\tau n}-(\bar{t}-m)}$. 
 By Corollary \ref{cor.nombre de zeros avec B=R=1} we deduce that 
$$
\mathcal{Z}(f_\lambda)
\le 
5\widetilde{\nu}_{\ell_e}\log2 + 5\log\left( 2+\frac{4m^{\frac{m+1}{2}}}{3\vert \alpha_0\vert^{\bar{t}-m}\vert a_0\vert^{\overline{\tau n}-(\bar{t}-m)}}  \right).
$$
Notice that since $m\ge 2$ and by Cauchy estimate $\vert \alpha_0\vert $
and $\vert a_0\vert $ are bounded from above by $1$, we have 
$$
\frac{m^{\frac{m+1}{2}}}{\vert \alpha_0\vert^{\bar{t}-m}\vert a_0\vert^{\overline{\tau n}-(\bar{t}-m)}}\ge 2\sqrt{2},
$$
and therefore
\begin{align*}
    2+\frac{4}{3}
    \frac{m^{\frac{m+1}{2}}}{\vert \alpha_0\vert^{\bar{t}-m}\vert a_0\vert^{\overline{\tau n}-(\bar{t}-m)}}
    \le \frac{3\sqrt{2}+8}{6}\frac{m^{\frac{m+1}{2}}}{\vert \alpha_0\vert^{\bar{t}-m}\vert a_0\vert^{\overline{\tau n}-(\bar{t}-m)}}.
\end{align*}
Recall that the constant \(5\) is in fact an approximation from above of \(1/\log(5/4)\). Since we have
\begin{align*}
    \frac{\log((3\sqrt{2}+8)/6)}{\log(5/4)} < 5\log 2,
\end{align*}
we can get the following bound from above for \(\mathcal{Z}(f_\lambda)\)
\begin{align*}
   \mathcal{Z}(f_\lambda) & \le  
   5\log2(\widetilde{\nu}_{\ell_e}+1) 
   + \frac{5}{2}(m+1)\log m\\
  & +5(\bar{t}-m)\log\left( \frac{1}{\vert \alpha_0 \vert}\right)
   +5(\overline{\tau n}-(\bar{t}-m))\log\left( \frac{1}{\vert a_0 \vert}
   \right). 
\end{align*}
    Applying the second part of Corollary \ref{cor.nombre de zeros avec B=R=1} we immediately obtain that the maximal number of zeros of $f_\lambda$, with respect to $\lambda\not=0$, is at most 
    $\widetilde{\nu}_{\ell_e}$
in $\bar{D}_\rho$, where
$$
\rho  =
\frac{1}{8^{\widetilde{\nu}_{\ell_e}}}
 \frac{3\vert \alpha_0\vert^{\bar{t}-m}\vert a_0\vert^{\overline{\tau n}-(\bar{t}-m)}}{4m^{\frac{m+1}{2}}} .
$$
    
    \item
In the subcase where $\nu_1=\nu_2$, by Proposition \ref{prop.valeur de sigma} we still have 
$\sigma=m$, 
       by \eqref{eq.valeur de b multiplicites egales} we have 
$b=\bar{\nu}_{\ell_e}+(k+1)(\tau_{\ell_e}-1)$, and by \eqref{eq.valeur de delta multiplicites egales}
$\delta
\ge \bar{C} \vert \mu\vert^{\bar{t}-\bar{\tau}}
\vert a_0\vert^{  \overline{\tau n}-\bar{\tau}'} 
\vert \bar{\delta}_k\vert^{\bar{\tau}' }$.  
Then Corollary \ref{cor.nombre de zeros avec B=R=1} gives immediately the bounds in our statement for $ \mathcal{Z}(f_\lambda)$ and $\rho$.\qedhere
\end{enumerate}
\end{proof}

In fact under condition~\ref{eq.Lacunarity Condition 1}, we can directly improve the bounds from the previous Theorem by replacing \(\sigma\) by \(\tau_{\ell_d}\) and the bounds on \(\delta\) by the bounds given in the second part of Proposition~\ref{prop.valeur de delta}.
\begin{thm}\label{thm.Bezout Bound for lacunary polynomials cas L1} Let a family $f_\lambda$ be given as in \eqref{eq:f.lambda}, with $R=B_R=1$.
Under condition \eqref{eq.Lacunarity Condition 1}, denoting by  $\mathcal{Z}({f_\lambda})$ the maximal number of zeros of $f_\lambda$ in $\bar{D}_{\frac{1}{4}}$, with respect to $\lambda\not\in L_b$, we have, for $P_1$ and $P_2$ not proportional,  
\begin{enumerate}
    \item
   in the case where $\nu_1\not=\nu_2$, 
    \begin{align*}
    \mathcal{Z}(f_\lambda)
    & \le 
5\widetilde{\nu}_{\ell_d}\log2 + 5\log\left( 2+\frac{4\tau_{\ell_d}^{\frac{\tau_{\ell_d}+1}{2}}}{3\vert \alpha_0\vert^{\bar{t}_{\ell_d}-\tau_{\ell_d}}\vert a_0\vert^{\tau_{\ell_d}d-(\bar{t}_{\ell_d}-\tau_{\ell_d})}}      \right) \\
    & = O(d\tau_{\ell_d}),
    \end{align*}
    and the maximal number of zeros of $f_\lambda$, with respect to $\lambda\not=0$, is at most 
    $\widetilde{\nu}_{\ell_d}$
in $\bar{D}_\rho$, where
$$
\rho  = 
\frac{3}{4\cdot 8^{\widetilde{\nu}_{\ell_d}}}
 \frac{\vert \alpha_0\vert^{\bar{t}_{\ell_d}-\tau_{\ell_d}}\vert a_0\vert^{\tau_{\ell_d}d-(\bar{t}_{\ell_d}-\tau_{\ell_d})}}{\tau_{\ell_d}^{\frac{\tau_{\ell_d}+1}{2}}}.
 $$
         \item
         In the case where $\nu_1=\nu_2$, 
         \begin{align*}
          \mathcal{Z}(f_\lambda) & \le    5(\nu d+(k+1)(\tau_{\ell_d}-1))\log2 \\
          & + 5\log\left( 2+\frac{4\tau_{\ell_d}^{\frac{\tau_{\ell_d}+1}{2}}} { 3C_{\ell_d} \vert \mu\vert^{\bar{t}_{\ell_d}-\frac{\tau_{\ell_d}(\tau_{\ell_d}+1)}{2}}
\vert a_0\vert^{\tau_{\ell_d}d - \frac{\tau_{\ell_d}(\tau_{\ell_d}-1)}{2}} 
\vert \bar{\delta}_k\vert^{\frac{\tau_{\ell_d}(\tau_{\ell_d}-1)}{2} }}  \right) \\
    & = O(d\tau_{\ell_d}),
         \end{align*}
 and the maximal number of zeros of $f_\lambda$, with respect to $\lambda\not=0$, is at most 
    $\nu d$
in $\bar{D}_\rho$, where
$$
\rho  = \frac{3}{4\cdot  8^{\nu d+(k+1)(\tau_{\ell_d}-1) } } \min\left\{\frac{2}{3}, 
 \frac{C_{\ell_d} \vert \mu\vert^{\bar{t}_{\ell_d}-\frac{\tau_{\ell_d}(\tau_{\ell_d}+1)}{2}}
\vert a_0\vert^{\tau_{\ell_d}d - \frac{\tau_{\ell_d}(\tau_{\ell_d}-1)}{2}} 
\vert \bar{\delta}_k\vert^{\frac{\tau_{\ell_d}(\tau_{\ell_d}-1)}{2} }}
{\tau_{\ell_d}^{\frac{\tau_{\ell_d}+1}{2}}} 
\right\}.
$$

\end{enumerate}

\end{thm}

\begin{rem}
The case where $P_1$ and $P_2$ are proportional has no practical interest, since in this situation $(P_1,P_2)$ parametrizes a line throughout the origin. Nevertheless, by sake of exhaustiveness, we give here the form of the bound obtained when \(\ell_d \ge 1\):
\begin{align*}
   \mathcal{Z}(f_\lambda) & \le
   5\log2(\nu d  +1) 
   + \frac{5}{2}(\ell_d+2)\log (\ell_d+1)\\
  & +5(\bar{t}'_1-\ell_d-1)\log\left( \frac{1}{\vert \mu \vert }\right)
   +5\bar{n}\log\left( \frac{1}{\vert a_0 \vert}
   \right)
    . 
\end{align*}
 and the maximal number of zeros of $f_\lambda$, with respect to $\lambda\not\in L_b$, is at most $\nu d$
in $\bar{D}_\rho$, where
$$
\rho  =
\frac{3}{4\cdot 8^{\nu d}}  \frac{\vert\mu\vert^{\bar{t}'_1 - (\ell_d+1)} 
\vert a_0\vert^{\bar{n}}}{(\ell_d+1)^{\frac{\ell_d+2}{2}}}  .
$$
Indeed, in this case  by Proposition \ref{prop.valeur de sigma} we have
$\sigma=\ell_d+1$, 
        by \eqref{eq.valeur de b P1 et P2 proportionnels} we have 
$b=\nu d$, and by \eqref{eq.valeur de delta P1 et P2 proportionnels}
$\delta
\ge 
\vert\mu\vert^{\bar{t}'_1 - (\ell_d+1)} 
\vert a_0\vert^{\bar{n}}$. 
Since $\vert a_0 \vert\le 1$, $\vert \mu a_0\vert=\vert \alpha_0 \vert \le 1$, and $\bar{n}\ge \bar{t}'_1 - (\ell_d+1) $, we have 
$\delta\le 1$, and thus
 $$
\frac{(\ell_d+1)^{\frac{\ell_d+2}{2}}}{\vert\mu\vert^{\bar{t}'_1 - (\ell_d+1)} 
\vert a_0\vert^{\bar{n}}}\ge 2\sqrt{2}.
$$
Again, applying Corollary \ref{cor.nombre de zeros avec B=R=1}, the same computation as in the case where $\nu_1\not=\nu_2$ gives the bounds for $ \mathcal{Z}(f_\lambda)$ and~$\rho$.

Of course, we can again make these bounds more accurate under the condition~\eqref{eq.Lacunarity Condition 1} using~
\eqref{eq.valeur de delta P1 et P2 proportionnels cas L1}.
\end{rem}

\begin{rem}\label{rem.dependance en a0}
Let $P=(P_1,P_2)$ be given, let us fix a lacunarity diagram (see Definition \ref{def.lacunarity diagram})  $(d,\ell_d, n_0, \ldots, n_{\ell_d},\tau_0, \ldots, \tau_{\ell_d},t_{0,1},\ldots t_{\ell_d,\tau_{\ell_d}})$, and let us choose a particular polynomial 
$Q_{\lambda_0}$ with this diagram. 
The polynomial $f_{\lambda_0}=Q_{\lambda_0}(P_1,P_2)$ being of degree $dD$, it has $dD$ roots in a disc $\bar{D}_R\subset \CC$, for some $R>0$. Following Remark \ref{rem.R=B=1}, the polynomial mapping 
$$
\widehat{P}(z)=P(4Rz)/\max_{z\in \bar{D}_1, i=1,2}\vert P_i(4Rz)\vert
$$
is bounded by $1$ on $\bar{D}_1$, and has $dD$ zeros in $\bar{D}_{\frac{1}{4}}$. It follows that the bounds given by Theorem \ref{thm.Bezout Bound for lacunary polynomials} for $\mathcal{Z}(f_\lambda)$, relatively to the particular  polynomial $\widehat{P}$ and relative to the particular degree $d$, 
are necessarily bigger than the largest expected bound $dD$, and consequently, for the particular polynomial $\widehat{P}$ and the particular degree $d$, have no interest.  This transformation of $P$ into $\widehat{P}$ shows that in order to obtain an interesting bound for $\mathcal{Z}(f_\lambda)$, uniform in the coefficients $\lambda$ of $Q_\lambda$, one cannot avoid a dependency on some coefficients of $P_1$ and $P_2$.

We stress here the fact that our bounds are not only uniform with respect to the coefficients of the lacunary polynomial $Q_\lambda$, and explicitly depend on its lacunarity diagram, but also depend on $\vert \alpha_0 \vert$ and $\vert a_0 \vert$, quantities which are modified when one replaces $P$ by $\widehat{P}$: considering 
$P$ with $\vert a_0\vert$ or $\vert \alpha_0\vert$ too small gives too large bounds (see for instance Example \ref{exam.Bounds} below).   

\end{rem}

The interest of the bounds provided by Theorem \ref{thm.Bezout Bound for lacunary polynomials} lies in particular, for a fixed choice of $P$, in the asymptotics they provide with respect to lacunarity diagrams, and in particular when $d$ goes to $\infty$. We give below such instance of interesting asymptotic bounds. 

\begin{exam}\label{exam.Bounds}
Let us fix $P_1,P_2\in \CC[X]$, of degree $D_1$ and $D_2$ (with \(D = \max(D_1,D_2)\)) and of multiplicity at least \(\nu \ge 1\) on zero. 
We then choose a sequence (parametrized by $d\in \NN$) of lacunarity diagrams 
$$
(\mathcal{D}_d)_{d\in \NN}=
\left( d,\ell_{d}, n_0, \ldots, n_{\ell_{d}}=d,\tau_0, \ldots, \tau_{\ell_d},t_{0,1},\ldots 
t_{ \ell_{d},\tau_{\ell_{d} }}
\right)_{d\in \NN}, 
$$
in the following way. 

\begin{itemize}
    \item[-] 
First of all we choose the sequence of degrees $n_0, n_1, \ldots, n_{\ell_d}=d$ so as to satisfy our simplest lacunarity condition 
\eqref{eq.Lacunarity Condition 1} 
$$n_0=0, n_1=1, \ldots, 
n_\ell=(D+1)^{\ell-1}, \ldots, n_{\ell_d}=(D+1)^{\ell_d-1}=d.$$ 
In particular $\ell_d=1+\dfrac{\log(d)}{\log(D+1)}.$ 
\item[-] 
Then, we assume that \(\tau_{\ell_d}\) is bounded by an integer \(\tau\) independent of~$d$.
\end{itemize}
 
It follows, by Theorem \ref{thm.Bezout Bound for lacunary polynomials cas L1} 
\eqref{item.Bezout Bound 1}, in case $\nu_1\not=\nu_2$, that
\begin{align}
\mathcal{Z}(f_\lambda) & \le  
 5\log2+ 5 d\max(\nu_1,\nu_2)\log2 \nonumber + \frac{5}{2}
   \left(\tau+1\right)\log \tau \nonumber\\
   & +5\left(\tau d-\frac{\tau(\tau-1)}{2}\right)\log\left( \frac{1}{\vert \alpha_0 \vert}\right)
   +5\left(\tau d - \frac{\tau(\tau-1)}{2}\right)\log\left( \frac{1}{\vert a_0 \vert}
   \right), \nonumber\\
  &\label{exam.bound1} \underset{d\to +\infty}\sim
  5 d\left(
  \max(\nu_1,\nu_2)\log2 +  \tau
  \log\left( \frac{1}{\vert \alpha_0 a_0 \vert}
  \right)\right).
\end{align}
As stated in Remark~\ref{rem.symmetry}, we observe in this example the symmetric role of \(P_1\) and \(P_2\).

Still by Theorem \ref{thm.Bezout Bound for lacunary polynomials cas L1} 
\eqref{item.Bezout Bound 1}, in case $\nu_1\not=\nu_2$, the maximal number of zeros of $f_\lambda$
 is at most $d\max(\nu_1,\nu_2)$ in the disc $\bar{D}_\rho$, with 

$$
\rho  \ge
\frac{3}{4\cdot 8^{d \max(\nu_1,\nu_2)}}
 \frac{\vert \alpha_0 a_0\vert^{\tau d - \frac{\tau(\tau-1)}{2}}}{\tau^{\frac{\tau+1}{2}}}.
 $$

 Noticing that $C_{\ell_d} \ge 1$, in case $\nu_1=\nu_2=\nu$, it follows from Theorem \ref{thm.Bezout Bound for lacunary polynomials cas L1} 
\eqref{item.Bezout Bound 2}, that our bound on $\mathcal{Z}(f_\lambda)$ has the same asymptotics, as $d\to +\infty$, as in the case $\nu_1\not=\nu_2$, namely

\begin{align}\label{exam.bound2} 
 %
  5 d\left(
  \nu\log2 +  \tau
  \log\left( \frac{1}{\vert \alpha_0 a_0 \vert}
  \right)\right).
 %
\end{align}

\end{exam}

\begin{rem}\label{rem.faster lacunarity}
The asymptotics \eqref{exam.bound1} and \eqref{exam.bound2} for our bound on $\mathcal{Z}(f_\lambda)$, given in Example \ref{exam.Bounds}, are better than the maximal number of zeros $dD$, only in case  
$$
   \tau
  \log\left( \dfrac{1}{\vert \alpha_0 a_0 \vert}
  \right) < \dfrac{D}{5}-\max(\nu_1,\nu_2)\log2 .
$$
As already noticed in Remark \ref{rem.dependance en a0}, in the situation where $\vert \alpha_0 a_0\vert $  
is too small, those bounds have no interest; They depend on a specific choice of $(P_1,P_2)$.
\end{rem}

\begin{rem}\label{rem.Jensen}
The classical Jensen formula estimating from above the number of zeros on $\bar{D}(0,1/4)$, for an analytic function $f$ with initial Taylor monomial $w_0z^k$, gives the bound 
$$k\left(2+\frac{\log(1/\vert w_0 \vert)}{\log 4}\right)+
\frac{1}{\log(4)}\log\left(\frac{\sup_{z\in \partial\bar{D}(0,1)}\vert f(z)\vert}{\vert w_0\vert }\right).
$$
 Applied to the family $f_\lambda$, from this estimate we cannot deduce a uniform bound with respect to $\lambda$, since, in this case, $w_0$ is a linear form in $\lambda$. 
\end{rem}

\section{The case of rational curves}\label{Section6}

We show in this last section how to reduce the case of rational parametrizations, the only situation we considered so far, to the case of polynomial parametri\-za\-tions, in order to also get a version of Theorem \ref{thm.Bezout Bound for lacunary polynomials} in the general case of rational planar curves. 

For this, let us consider now the general case of the intersection of the algebraic curve \(Q=0\) of $\CC^2$ and of a rational curve is \((P_1/V,P_2/V)\), where \(P_1\), \(P_2\), \(V\) are in \(\CC[X]\), $X$ divides $P_1$ and $P_2$, \(\gcd(P_1,V)=\gcd(P_2,V)=1\), and \(P_1,P_2\) are not proportional. In particular, \(V\) has a nonzero constant term, denoted \(\beta_0\). 

For bounding the number of zeros of \(Q(P_1/V,P_2/V)\) following Section~\ref{section 4}, we are reduced to the case where the rational curve is defined on \(\bar{D}_1\) and both functions \(P_1/V\) and \(P_2/V\) are bounded by \(1\) on \(\bar{D}_1\) (see Remark~\ref{rem.R=B=1}). By Cauchy's estimate, \(\left\lvert a_0/\beta_0 \right\rvert \leq 1\) and \(\left\lvert \alpha_0/\beta_0 \right\rvert \leq 1\).

As above we set \(f_{\lambda}(z) = Q_{\lambda}(P_1(z),P_2(z))\). Then the Taylor series at the origin of \(g_{\lambda}(z) = Q_{\lambda}((P_1/V,P_2/V)(z))\) has for coefficients linear forms \(v_k, k\in \NN\), in the parameters \(\lambda_1,\ldots,\lambda_m\), and we write
\[
    g_{\lambda}(z) = \sum_{k \ge 0} v_k(\lambda) z^k, \quad \text{where }v_k(\lambda) = \sum_{i=1}^m c_{k}^i \lambda_i.
\]

With the notation of the previous sections, for $\ell\in \llbracket 1, \ell_d \rrbracket $, let 
\[P_1^{I}P_2^{n_{\ell}-I}(z) = \sum_{k\ge 0} r_k z^k 
\ \mathrm{and}
\ 1/V^{n_{\ell}}(z) = \sum_{k \ge 0} s_k z^k \] 
the Taylor series at the origin of $P_1^{I}P_2^{n_{\ell}-I}$ and 
$1/V$. 
In particular \(s_0 = 1/(\beta_0)^{n_{\ell}}\). 
We then have
\begin{align}
    \dfrac{P_1^I P_2^{n_{\ell}-I}}{V^{n_{\ell}}}(z) & = \sum_{k \ge 0} \sum_{j=0}^k s_jr_{k-j} z^k \nonumber\\
    & = \sum_{j \ge 0} s_j \sum_{k \ge 0} r_{k}z^{k+j} \label{eq:decomposition_series}.
\end{align}

Still with the notation of the beginning of Section \ref{Section 5}, let us consider the blocks \(M(f_\lambda)=M_{n_{\ell},T_{\ell}}(f_{\lambda})\) and \(M(g_\lambda)=M_{n_{\ell},T_{\ell}}(g_{\lambda})\) of the Bautin matrix of \(f_{\lambda}\) and \(g_\lambda\) respectively. The number of  rows of \(M(g_{\lambda})\) is potentially infinite, and we only consider the \(b(g_{\lambda})\) first rows of \(M(g_{\lambda})\). By~\eqref{eq:decomposition_series}, \(M(g_{\lambda})\) is a linear combination of several shifts of \(M(f_{\lambda})\). More precisely, by denoting by \(\widetilde{M}_{j}\) the \(j\)-th shift of \(M(f_{\lambda})\), obtained by adding \(j\) rows of zeros above \(M(f_{\lambda})\), we have that \(M(g_{\lambda})=\sum_{j \ge 0} s_j \widetilde{M}_j\). In particular, if we consider the triangular matrix 
\[
    G_{b,n_{\ell},T_{\ell}} = \begin{pmatrix}
        s_0 & 0 & \cdots &0 \\
        s_1 & s_0 & & \vdots\\
        \vdots & \vdots & \ddots & 0 \\
        s_b & s_{b-1} & \cdots & s_0
    \end{pmatrix},
\] 
we get that \(M(g_{\lambda}) = G_{b,n_{\ell},T_{\ell}} \cdot \widetilde{M}(f_{\lambda}) \) where \(\widetilde{M}(f_{\lambda})\) is the matrix \(M(f_{\lambda})\) truncated or completed with sufficiently many rows of zeros. As \(s_0\) is nonzero, the matrix \(G_{b,n_{\ell},T_{\ell}}\) is inversible. Hence, the parameters \(b\) and \(\sigma\) are no altered by this multiplication and the minor \(\delta\) is multiplied by \(s_0^{\tau_\ell}\).

Therefore, the matrix obtained by multiplying the full Bautin matrix 
 \(M_{T_{0},\ldots,T_{\ell_d}}(g_{\lambda})\), of the family $g_\lambda$, on the left by the matrix
\[
   \begin{pmatrix}
    0 & \cdots & \cdots & 0 \\
    G_{b,n_0,T_0}  &&& \vdots\\
    & G_{b,n_1,T_1} && \\
    \vdots & & \ddots & \vdots \\
    &&& G_{b,n_{\ell},T_{\ell}} \\
    0 &\cdots&\cdots& 0
   \end{pmatrix},
\]
where there are exactly \(\bar{\nu}_{\ell}\) rows of zeros above the block $G_{b,n_{\ell},T_{\ell}}$, have the same pertinent blocks, from the rank point of view, as \(M_{T_{0},\ldots,T_{\ell_d}}(f_{\lambda})\). We recall that \(\bar{\nu}_\ell\) was defined in section~\ref{Section 5} and represents the number of rows of zeros above the block \(M_{n_{\ell},T_{\ell}}\).

Consequently, under conditions~\eqref{eq.Lacunarity Condition 1}, \eqref{eq.Lacunarity Condition 1bis}, and~\eqref{eq.Lacunarity Condition 1bis with mult}, Propositions~\ref{prop.valeur de sigma} (value of \(\sigma\)) and~\ref{prop.valeur de b} (value of \(b\)) still stand. Concerning Proposition~\ref{prop.valeur de delta}, we now have that for \(\nu_1\ne \nu_2\)
\begin{equation*}
\delta\ge \left\lvert \dfrac{\alpha_0}{\beta_0}\right\rvert^{\bar{t}-m
}
\left\lvert \dfrac{a_0}{\beta_0}\right\rvert^{\overline{\tau n}-(\bar{t}-m)
} ,
\end{equation*}
and for $\nu_1=\nu_2$  
\begin{equation*}
\delta\ge
\bar{C} 
\vert \mu\vert^{\bar{t}-\bar{\tau}}
\left\lvert \dfrac{a_0}{\beta_0}\right\rvert^{  \overline{\tau n}-\bar{\tau}'  } 
\left\lvert \dfrac{\bar{\delta}_k}{\beta_0} \right\rvert^{
\bar{\tau}' }.
\end{equation*}

Similarly to Theorem~\ref{thm.Bezout Bound for lacunary polynomials}, using \(\left\lvert a_0/\beta_0 \right\rvert \leq 1\) and \(\left\lvert \alpha_0/\beta_0 \right\rvert \leq 1\), we obtain that under conditions~\eqref{eq.Lacunarity Condition 1}, \eqref{eq.Lacunarity Condition 1bis}, and~\eqref{eq.Lacunarity Condition 1bis with mult}, for \(m \ge 2\), the maximal number of zeros of \(g_{\lambda}\) in \(D_{1/4}\) is bounded
\begin{itemize}
    \item[-] in case \(\nu_1\ne \nu_2\), by
  \begin{align*}
   \mathcal{Z}(f_\lambda) & \le  
   5(\widetilde{\nu}_{\ell_d}+1)\log2 
   + \frac{5}{2}(m+1)\log m\\
  & +5(\bar{t}-m)\log\left( \frac{\vert \beta_0 \vert}{\vert \alpha_0 \vert}\right)
   +5(\overline{\tau n}-(\bar{t}-m))\log\left( \frac{\lvert \beta_0\rvert}{\vert a_0 \vert}
   \right), 
\end{align*}
\item[-] and in case \(\nu_1= \nu_2\), by
 \begin{align*}
          \mathcal{Z}(f_\lambda) & \le    5(\nu d+(k+1)\tau_{\ell_d}-k)\log2 \\
          & + 5\log\left( 2+\frac{4m^{\frac{m+1}{2}}}{ 3\bar{C} \vert \mu\vert^{\bar{t}-\bar{\tau}}
\vert a_0/\beta_0 \vert^{  \overline{\tau n}-\bar{\tau}'} 
\vert \bar{\delta}_k /\beta_0\vert^{\bar{\tau}' } }  \right).
         \end{align*}
\end{itemize}

Finally, the bound can be again simplified under condition~\ref{eq.Lacunarity Condition 1} when \(\tau_{\ell_d} \ge 2\)
\begin{align*}
   \mathcal{Z}(f_\lambda) & \le  
   5(\widetilde{\nu}_{\ell_d}+1)\log2 
   + \frac{5}{2}(\tau_{\ell_d}+1)\log (\tau_{\ell_d})\\
  & +5(\bar{t}_{\ell_d}-\tau_{\ell_d})\log\left( \frac{\vert \beta_0\vert}{\vert \alpha_0 \vert}\right)
   +5(\tau_{\ell_d}d- (\bar{t}_{\ell_d}-\tau_{\ell_d}))\log\left( \frac{\vert \beta_0\vert}{\vert a_0 \vert}
   \right).
\end{align*}

\vfill\eject

\bibliographystyle{siam}
\bibliography{Lacunarity}

\begin{thebibliography}{10}

\bibitem{Bau1}
{\sc N.~N. Bautin}, {\em Du nombre de cycles limites naissant en cas de
  variation des coefficients d'un \'etat d'\'equilibre du type foyer ou
  centre}, C. R. (Doklady) Acad. Sci. URSS (N. S.), 24 (1939), pp.~669--672.

\bibitem{Bau2}
\leavevmode\vrule height 2pt depth -1.6pt width 23pt, {\em On the number of
  limit cycles which appear with the variation of coefficients from an
  equilibrium position of focus or center type}, American Math. Soc.
  Translation, 1954 (1954), p.~19.

\bibitem{Bernstein}
{\sc D.~N. Bernstein}, {\em The number of roots of a system of equations.},
  Funkcional. Anal. i Prilo\v{z}en.,  (1975), pp.~1--4.

\bibitem{Bezout}
{\sc E.~B\'{e}zout}, {\em General theory of algebraic equations}, Princeton
  University Press, Princeton, NJ, 2006.
\newblock Translated from the 1779 French original by Eric Feron.

\bibitem{BCSS}
{\sc L.~Blum, F.~Cucker, M.~Shub, and S.~Smale}, {\em Complexity and real
  computation}, Springer-Verlag, New York, 1998.
\newblock With a foreword by Richard M. Karp.

\bibitem{Bur00}
{\sc P.~B{\"u}rgisser}, {\em Completeness and reduction in algebraic complexity
  theory}, vol.~7, Springer Science \& Business Media, 2000.

\bibitem{CoYo1}
{\sc G.~Comte and Y.~Yomdin}, {\em Zeroes and rational points of analytic
  functions}, Ann. Inst. Fourier (Grenoble), 68 (2018), pp.~2445--2476.

\bibitem{FraYom}
{\sc J.-P. Francoise and Y.~Yomdin}, {\em Bernstein inequalities and
  applications to analytic geometry and differential equations}, J. Funct.
  Anal., 146 (1997), pp.~185--205.

\bibitem{Hart}
{\sc R.~Hartshorne}, {\em Algebraic geometry}, Graduate Texts in Mathematics,
  No. 52, Springer-Verlag, New York-Heidelberg, 1977.

\bibitem{HTR}
{\sc H.~Hauser, J.-J. Risler, and B.~Teissier}, {\em The reduced {B}autin index
  of planar vector fields}, Duke Math. J., 100 (1999), pp.~425--445.

\bibitem{Hru13}
{\sc P.~Hrube{\v{s}}}, {\em On the real $\tau$-conjecture and the distribution
  of complex roots}, Theory of Computing, 9 (2013), pp.~403--411.

\bibitem{KhoFew}
{\sc A.~G. Khovanski\u{\i}}, {\em Fewnomials}, vol.~88 of Translations of
  Mathematical Monographs, American Mathematical Society, Providence, RI, 1991.
\newblock Translated from the Russian by Smilka Zdravkovska.

\bibitem{Koi11}
{\sc P.~Koiran}, {\em Shallow circuits with high-powered inputs}, in
  Proceedings of the Second Symposium on Innovations in Computer Science (ICS),
  2011, pp.~309--320.

\bibitem{KoiPoTa}
{\sc P.~Koiran, N.~Portier, and S.~Tavenas}, {\em On the intersection of a
  sparse curve and a low-degree curve: a polynomial version of the lost
  theorem}, Discrete Comput. Geom., 53 (2014), pp.~48--63.

\bibitem{KS20}
{\sc P.~Koiran and M.~Skomra}, {\em Intersection multiplicity of a sparse curve
  and a low-degree curve}, J. Pure Appl. Algebra, 224 (2020), pp.~106279, 16.

\bibitem{Kush1}
{\sc A.~G. Kushnirenko}, {\em Poly\`{e}dres de {N}ewton et nombres de
  {M}ilnor}, Invent. Math., 32 (1976), pp.~1--31.

\bibitem{Kush}
\leavevmode\vrule height 2pt depth -1.6pt width 23pt, {\em Newton polytopes and
  the {Bezout} theorem}, Funct. Anal. Appl., 10 (1977), pp.~233--235.

\bibitem{Kushletter}
\leavevmode\vrule height 2pt depth -1.6pt width 23pt, {\em A question about a
  conjecture,
  \href{https://franksottile.github.io/research/pdf/Kushnirenko.pdf}{link}},
  Letter to {F}. {S}ottile,  (2008).

\bibitem{Rou}
{\sc R.~Roussarie}, {\em A note on finite cyclicity property and {H}ilbert's
  16th problem}, in Dynamical systems, {V}alparaiso 1986, vol.~1331 of Lecture
  Notes in Math., Springer, Berlin, 1988, pp.~161--168.

\bibitem{Roy.Yom}
{\sc N.~Roytwarf and Y.~Yomdin}, {\em Bernstein classes}, Ann. Inst. Fourier
  (Grenoble), 47 (1997), pp.~825--858.

\bibitem{SY10}
{\sc A.~Shpilka and A.~Yehudayoff}, {\em Arithmetic circuits: A survey of
  recent results and open questions}, Foundations and Trends{\textregistered}
  in Theoretical Computer Science, 5 (2010), pp.~207--388.

\bibitem{Sottile}
{\sc F.~Sottile}, {\em Real solutions to equations from geometry}, vol.~57 of
  Univ. Lect. Ser., Providence, RI: American Mathematical Society (AMS), 2011.

\bibitem{VdP}
{\sc A.~J. Van~der Poorten}, {\em On the number of zeros of functions},
  Enseignement Math. (2), 23 (1977), pp.~19--38.

\bibitem{Yako}
{\sc S.~Yakovenko}, {\em On zeros of functions from {B}ernstein classes},
  Nonlinearity, 13 (2000), pp.~1087--1094.

\bibitem{Yom}
{\sc Y.~Yomdin}, {\em Oscillation of analytic curves}, Proc. Amer. Math. Soc.,
  126 (1998), pp.~357--364.

\end{thebibliography}

\end{document}